\documentclass[arxiv,final]{elsarticle}

\usepackage{mathrsfs}
\usepackage{enumerate}
\usepackage[utf8]{inputenc}
\usepackage{algpseudocode}
\usepackage{siunitx}

\journal{Journal of Computational and Applied Mathematics}

\renewcommand{\leq}{\leqslant}
\renewcommand{\geq}{\geqslant}
\renewcommand{\le}{\leq}
\renewcommand{\ge}{\geq}

\newcommand\TheTitle{Finite element model updating for structural applications}

\makeatletter
\@ifclasswith{elsarticle}{arxiv}{
	\def\ps@pprintTitle{%
		\let\@oddhead\@empty
		\let\@evenhead\@empty
		\def\@oddfoot{\centerline{\thepage}}%
		\let\@evenfoot\@oddfoot}
}{}
\makeatother

\title{\TheTitle}

\usepackage{amsmath,amssymb,amsthm}
\usepackage{algorithm}
\usepackage{mathtools}
\usepackage{tikz}
\usepackage{pgfplots}
\usepackage{geometry}
\usepackage{color}
\usepackage{bbm}

\usetikzlibrary{patterns}

\newcommand{\fr}{\phi^{\cal R}}
\newcommand{\f}{\phi}
\newcommand{\IR}{\mathbb R}

\usepackage{amsopn}
\DeclarePairedDelimiter{\norm}{\lVert}{\rVert}

\newcommand{\matlab}{\texttt{MATLAB}}

\theoremstyle{plain}
\newtheorem{lemma}{Lemma}[section]
\newtheorem{theorem}[lemma]{Theorem}

\theoremstyle{remark}
\newtheorem{remark}[lemma]{Remark}

\theoremstyle{definition}

\renewcommand{\leq}{\leqslant}
\renewcommand{\geq}{\geqslant}
\renewcommand{\tilde}{\widetilde}

\begin{document}

	\begin{frontmatter}
		
		\author[ISTI]{Maria Girardi\fnref{moscardo}}
		\ead{maria.girardi@isti.cnr.it}
		
		\author[ISTI]{Cristina Padovani\fnref{moscardo}}
		\ead{cristina.padovani@isti.cnr.it}
		
		\author[ISTI]{Daniele Pellegrini\fnref{moscardo}}
		\ead{daniele.pellegrini@isti.cnr.it}
		
		\author[ISTI,UNIFI]{Margherita Porcelli}
		\ead{margherita.porcelli@unifi.it}
		
		\author[ISTI]{Leonardo Robol\corref{cor1}\fnref{moscardo}}
		\ead{leonardo.robol@isti.cnr.it}
		
		\fntext[moscardo]{This research has been partially supported by the
			Region of Tuscany (Project ``MOSCARDO - ICT technologies for
			structural monitoring of age-old constructions based on wireless
			sensor networks and drones'', 2016--2018, FAR FAS), and by the
			GNCS/INdAM project ``Metodi numerici avanzati per equazioni e
			funzioni di matrici con struttura''.}
		
		\address[ISTI]{Institute of Information Science and
			    Technologies ``A. Faedo'', ISTI-CNR, Pisa, Italy.}
		
		\address[UNIFI]{Dipartimento di Ingegneria Industriale, %
			    Università degli Studi di Firenze, Italy,}
		    
		    \cortext[cor1]{Corresponding author}
		
		\title{\TheTitle}

\begin{abstract}
	A novel method for performing model updating on finite element models is presented. 
	The approach is particularly tailored to modal analyses of buildings, by which
	the 
	lowest frequencies, obtained by using sensors and system identification
	approaches, need to be matched to the numerical ones predicted by the model. 
	This is done by optimizing some unknown material parameters (such as mass density and Young's modulus)
	of the materials and/or the boundary conditions, 
	which are often known only approximately. In particular, this is the case
	when considering historical buildings.
	
	The straightforward application of a general-purpose optimizer can be 
	impractical, given
	the large size of the model involved. In the paper, we show that, by slightly modifying the projection
	scheme used to compute the eigenvalues at the lowest end of the spectrum 
	one can obtain  local parametric reduced order models
	that, embedded in a trust-region scheme, form the basis for a reliable and efficient specialized algorithm. 
	
	We describe an optimization strategy based on this approach, and we provide
	numerical experiments that confirm its effectiveness and accuracy.
\end{abstract} 
		
		
	\end{frontmatter}

\section{Introduction}
\label{sec:introduction}

Finite element (FE) model updating is a procedure aimed at calibrating the
FE model of a structure in order to match numerical and
experimental results. Introduced in the 1980s, it turned out to play
a crucial role in the design, analysis and maintenance of aerospace,
mechanical and civil engineering structures \cite{chen1980analytical,douglas1982dynamic,friswell2013finite,marwala2010finite}.
In structural mechanics, model updating techniques are used in
conjunction with vibrations measurements to determine unknown system
characteristics, such as the materials’ properties, constraints, etc.
The resulting updated FE model can then be used to obtain
reliable predictions on the dynamic behavior of the structure
subjected to time-dependent loads. A further important application
of model updating, within the framework of structural health
monitoring, is damage detection 
\cite{simoen2015dealing,teughels2005damage}. Within this framework, 
damage can be identified
based on the
assumption that its presence is associated with a
decrease in the stiffness of some elements, with consequent changes
to the structure's modal characteristics.

Finite element model updating involves the solution of a constrained
optimization problem, whose objective function is generally expressed as
the discrepancy between experimental and numerical quantities, such
as the structure's natural frequencies and mode shapes. The constraints
are given by the boundary conditions and other physical limitations to the
degrees of freedom involved. 
Ill-posedness or
ill-conditioning can affect model updating formulations and lead to numerical
problems due to inaccuracy in the model and lack of
information in the measurements. Among the methods aimed at
quantifying uncertainties, the probabilistic Bayesian approach is
one of the most adopted \cite{simoen2015dealing}.

Application of FE model updating to ancient masonry
buildings is relatively recent. In
\cite{aoki2007structural,araujo2012seismic,
	cabboi2017continuous,
	ceravolo2016vibration,compan2017structural,
	erdogan2017discrete,fragonara2017dynamic,
	gentile2007ambient,kocaturk2017investigation,
	perez2011characterization,
	ramos2011dynamic,ramos2010monitoring,%
    torres2017operational} a vibration-based model updating is conducted,
and preliminary FE models are tuned by using the dynamic
characteristics determined through system identification techniques.
In the papers cited above the modal analysis of the FE models are
conducted via commercial codes, and the model updating procedure is
implemented separately. In this paper, on the contrary, FE
model updating is integrated within a software package,
the NOSA-ITACA code \cite{nosa2017}, able to manage the large-scale
problems encountered in applications. In particular, the algorithms
for the solution of the constrained optimization problem integrated in
NOSA-ITACA exploit the structure of the stiffness and mass matrices
and the fact that only a few of the smallest eigenvalues have to be
calculated. This new procedure reduces both the total computation
time of the numerical process and user’s effort, thus
providing the scientific and technical communities with efficient
algorithms specific for FE model updating.

A simple form of FE model updating has been employed in
combination with the NOSA-ITACA code in 
\cite{azzara2016assessment,azzara2018influence,pellegrini2017anew},
to perform modal analyses of the San Frediano bell tower and the Clock
tower in Lucca, Italy. In these works the
optimal 
values of Young’s modulus and the materials' mass density
are determined by fitting the data measured by seismometric stations
placed on the towers and running several simulation on a grid 
of feasible values. However, this approach becomes impractical 
if the number of free parameters or the size of the model
is considerably increased. 

Model reduction techniques have long been used to
reduce the size of complex FE models to a more
manageable order. This allows, for example, performing more demanding,
and computationally costly,
numerical tasks, such as optimization, simulation, and so on, in an
efficient manner. In particular, reduced models are natural
candidates for optimization of the model's free parameters,
either to improve some engineering properties or to better match the
empirically measured modes. In practice, this step often turns into
an eigenvalue optimization, where part of the spectral structure is
tuned following some prescribed criteria.

However, when the model depends on parameters, it is generally
non-trivial to obtain a reduced parametric model that accurately
reflects the behavior of the original one for all possible parameter values. This is precisely the aim of so-called {parametric
	model reduction}, which has recently been used in
combination with optimization algorithms \cite{amsallem2015design,o2017computing,qian2017certified,%
zahr2015progressive},
especially involving 
trust-region solvers \cite{Biegler03,yue2013accelerating}. In this
work, we propose an efficient model reduction strategy derived
by properly recycling the Lanczos projection. Our approach is 
tailored to the needs of structural FE analysis. 

In fact, FE study of buildings has some peculiar
features that justify trying to devise an adapted method. As
discussed in the following, in this context the natural frequencies of
interest are those at the lowest end of the real spectrum.
In order to compute them accurately, the natural
choice is an (inverse) Lanczos method. When a parametric model is
given, the Lanczos projection can be re-interpreted as a
parameter dependent model reduction, whereby only the relevant
part of the spectrum is matched. We show that this step can be
performed efficiently, and that its combination with a trust-region
method allows matching the measured
frequencies with the ones predicted by the parametric model.

The paper is organized as follows. In Section~\ref{sec:model}
we briefly discuss the FE model for the structural applications
we are interested in, and formulate the optimization problem
related to model updating. In Section~\ref{sec:trustregion} we recall
the requirement for a model to be used in a trust-region
optimization method, in particular those required to
guarantee convergence. Section~\ref{sec:reducedmodel} 
describes how to construct a model that satisfies these requirements
by slightly modifying the Lanczos projection step used to compute
the smallest frequencies, and finally Section~\ref{sec:examples}
illustrates some 
tests of our implementation on some problems. In Section~\ref{sec:concluding-remarks}
we draw some conclusions 
and discuss future lines of research. 

\section{The model: modal analysis of masonry buildings and frequency matching}
  \label{sec:model}

Although the constituent masonry materials of historical buildings
exhibit different strengths under tension and compression and, therefore, 
behave nonlinearly, modal analysis, which is based on the
assumption that the materials are linear elastic, is widely used in
applications and provides important qualitative information on the
dynamic behavior of masonry structures.
Modal analysis consists in the solution of the constrained
generalized eigenvalue problem
\begin{align} \label{gep_MK}
K\, v &= \omega^2\, M\,v, & 
\text{subject to } \ \ C v &=0,
\end{align}
with $C\in \mathbb{R}^{h\times n}$ and $h\ll n$. The left part
of \eqref{gep_MK} is derived from the differential equation
\begin{equation}\label{edo}
 M \ddot{u} + K u = 0,
\end{equation}
governing the undamped free vibrations of a linear elastic structure
discretized into finite elements. In \eqref{edo} $u$ is the
displacement vector, which belongs to $\mathbb{R}^{n}$ and depends
on time $t$, $\ddot{u}$ is the second-derivative of $u$ with respect
to $t$, and $K$ and $M\in \mathbb{R}^{n\times n}$ are the stiffness
and mass matrices of the FE model. $K$ is symmetric
and positive-semidefinite, $M$ is symmetric and positive-definite,
and both are sparse and banded. Displacements $u_i$ are also called
degrees of freedom; the integer $n$ is the total number of degrees
of freedom of the system and is generally very large, since it
depends on the level of discretization of the problem. By assuming
that
\begin{equation}\label{Usin}
 u = v \cdot \sin(\omega t),
\end{equation}
with $v$ a vector of $\mathbb{R}^{n}$ and $\omega$ a real scalar,
and applying the mode superposition procedure
\cite{bathe1976numerical}, equation (\ref{edo}) is transformed into
the generalized eigenvalue problem (\ref{gep_MK}). The right part of equation~\eqref{gep_MK} expresses the fixed constraints and the master-slave
relations assigned to displacement $u$, written in terms of vector
$v$. Imposing the constraints and boundary conditions is equivalent
to projecting the matrices $K$ and $M$ on a subspace where they
are a symmetric positive definite pencil (the right kernel of the operator $C$). 
In the following, we assume that 
this projection has already been done  --- and we refer the reader
to \cite{porcelli2015solution} for further details. Notice, in particular,
that if $K$ and $M$ depend \emph{linearly}
on some parameters $\mathbf x := (x_1, \ldots, x_\ell)$, 
the same holds true for their projection. More generally, the smooth
dependency of $K$ and $M$ is preserved by the projection, since 
it is a linear operation. This will be exploited in the 
analysis of the parametric projection in Section~\ref{sec:reducedmodel}. 

The eigenvalues $\omega_i^2$ of (\ref{gep_MK}) are linked to the
natural frequencies, or eigenfrequencies $f_i$ of the structure
via the relation $f_i=\omega_i/2\pi$, and the eigenvectors $v^{(i)}$
are the corresponding mode shape vectors, or eigenmodes. Together
with the natural frequencies, the mode shapes provide a good deal of qualitative
information on the structure's deformations under dynamic loads.

Measuring the vibrations of masonry buildings is a common practice
for assessing their dynamic behavior. Historic constructions are
subjected to a number sources of vibrations, such as traffic,
micro--tremors, wind and earthquakes. The availability of
sensitive instruments to detect buildings' movements makes it
possible to conduct accurate, long--term monitoring campaigns. In
fact, ambient vibration monitoring can provide important information
on the structural health of old masonry constructions, as it is a
non-destructive technique able to capture the most important
features of their dynamic behavior, such as natural frequencies,
damping ratios, mode shapes and wave propagation velocities. Once
the influence of environmental factors has been accounted for,
changes in these dynamic properties over time may represent
effective structural damage indicators \cite{azzara2018influence}.

FE model updating is a procedure that enables 
calibrating a finite element model of a structure in order to match
numerical and experimental results. Thus, model updating techniques
are used in conjunction with vibrations measurements to determine
a structure's characteristic, such
as materials' properties, constraints, etc., which are generally unknown. 

\subsection{Formulation of the optimization problem} \label{sec:formulation-optimization}

The model updating problem can be reformulated as an optimization problem
by assuming that the (projected) stiffness and mass matrices $K$ and $M$
are functions of the parameter vectors $\mathbf x$. We use the
notation
\[
K := K(\mathbf x), \qquad
M := M(\mathbf x), \qquad
\mathbf x \in \mathbb R^\ell, 
\]
to denote this dependency. The set of valid choices for the parameters is 
denoted by $\Omega$, which we assume to be an $\ell$-dimensional
box, that is
\begin{equation} \label{eq:Omega}
\Omega = [ a_1 , b_1 ] \times \ldots \times [ a_\ell, b_\ell ], 
\end{equation}
for certain values $a_i < b_i$, $i = 1, \ldots, \ell$. We also assume
that initial estimates for the parameter values are available, and 
we denote these by $\mathbf x^{(0)}$. When this is not the case, we choose
$\mathbf x^{(0)}$ as the center of the box $\Omega$ (i.e., the vector with the midpoints 
of the intervals $[a_i, b_i]$ as coordinates). 

Our ultimate aim is to determine the optimal value of 
$\mathbf x$ that minimizes a certain cost functional $\phi(\mathbf x)$
within the box $\Omega$.
This is an instance of the optimization problem
\begin{equation}  
\label{eq:optpb}
\min_{\mathbf x \in \Omega} \phi(\mathbf x).
\end{equation}

The choice of the objective function $\phi(\mathbf x)$ is related
to the frequencies that we want to match. If we need to match 
$s$ frequencies
of the model, we choose a suitable weight vector
$\mathbf w := [ w_1 , \ldots, w_s ]$, with $w_i \geq 0$, 
and define the functional
$\phi(\mathbf x)$ as follows:
\begin{equation}  
\label{eq:obj}
\phi(\mathbf x) := 
\norm*{
	\frac{\sqrt{\Lambda_s(K, M)}}{2\pi} - \mathbf f }_{\mathbf w,2}^2, \qquad
\norm{\mathbf y}_{\mathbf w, 2} := \sqrt{\mathbf y^T D_{\mathbf w}^2 \mathbf y}, 
\end{equation}
where $\mathbf f$ is the vector of the measured frequencies,
$D_{\mathbf w} = \mathrm{diag}(w_1, \ldots, w_s)$
and $\Lambda_s(K, M)$ is the vector containing the smallest $s$
eigenvalues of the pencil $K - \lambda M$, 
ordered according to their magnitude. 

The vector $\mathbf w$ encodes the weights that should be given
to each frequency in the optimization scheme. If the aim is to 
minimize the distance between the vector of measured and computed 
frequencies in the usual Euclidean norm, then $\mathbf w = \mathbbm{1}_s$, 
the vector of all ones, should be chosen. 
If, instead, a relative accuracy on the frequencies is desired, 
$w_i = f_i^{-1}$ is the natural choice. To avoid scaling issues in the
objective function, we always normalize $\mathbf w$ to have 
Euclidean norm $1$.

\section{The optimization framework}
\label{sec:trustregion}

In this section we describe the optimization framework that we use. 
We choose to rely on a trust-region scheme, which is defined by building a sequence of 
local models for the objective
function. 

Construction of the local models is discussed in Section~\ref{sec:reducedmodel}, 
where we show that such models satisfy the minimal requirements for 
a well-defined  trust-region scheme. 
The practical choice of the objective function and the weights
has been deferred to Section~\ref{sec:examples}, 
where several options are discussed. 

\subsection{Trust region methods}
Trust-region methods are well-known iterative methods for solving 
bound-constrained optimization problems, see e.g. \cite{conn1988global,trBible,porcelli13}. 
Their main distinctive features are robustness and
convergence to first-order critical points regardless of the choice of 
the initial estimate.

We now review the basic steps in a trust-region algorithm for solving 
the
bound-constrained optimization problem of Equation~\eqref{eq:optpb}, 
where $\f:\IR^\ell \rightarrow \IR$ is a smooth function and
$\Omega$ is defined as in \eqref{eq:Omega}. 

The key idea of a trust-region method is to define, at iteration $k$, 
a model $\fr_k$ for the objective function $\f$ 
around the 
current  iterate $\mathbf x^{(k)}$, together with a 
region within such model can be trusted to provide an
adequate representation of $\f$.
The trial step $\mathbf s^{(k)}$ is then computed, either exactly
or approximately, by minimizing this model within the trust-region. 
The trust-region is the set of all points
$$
{\mathcal B}_k = \{\mathbf x \in \IR^\ell |\, \|\mathbf x - \mathbf x^{(k)}\|\le \Delta_k\},
$$
where $\Delta_k$ is called the trust-region radius and $\|\cdot\|$ is a  norm
equivalent to the Euclidean norm.

Due to the presence of bound-constraints, all the generated
iterates $\mathbf x^{(k)}$ must be ensured to be feasible, that is, $\mathbf x^{(k)} \in \Omega$ for all $k\ge 0$.
One strategy consists in minimizing the model within the set $\mathcal B_k \cap \Omega$.
When the feasible set $\Omega$ is a box, 
the trust-region is generally defined using the $\infty$ norm.

Having minimized the model on $\mathcal B_k \cap \Omega$, it must be decided
whether to accept the trial step or to change the trust-region radius.
Usually, the 
trust-region radius and the new point $\mathbf x^{(k)} + \mathbf s^{(k)}$ are tested simultaneously 
to assess the quality of the approximation yielded by the local model. This is
measured 
using the ratio $\rho^{(k)}$ between the actual and the predicted reduction defined as follows:
\begin{equation}\label{eq:ratio}
 \rho^{(k)} = \frac{\f(\mathbf x^{(k)})-\f(\mathbf x^{(k)}+\mathbf s^{(k)})}{\fr_k(\mathbf x^{(k)}) - \fr_k(\mathbf x^{(k)}+ \mathbf s^{(k)})}
\end{equation}
If $\rho^{(k)}$ is close to 1, there is good agreement between the 
model $\fr$ and the function $\f$ over this step, so $\mathbf x^{(k)} + \mathbf s^{(k)}$ is accepted 
as the new iterate and it is safe to increase the radius of
the trust-region for the next iteration.
If $\rho^{(k)}$ is positive but not close to 1, then $\mathbf x^{(k+1)} = \mathbf x^{(k)} + \mathbf s^{(k)}$ but the 
trust-region radius is not altered. If  $\rho^{(k)}$ is close to zero or
negative, step $\mathbf s^{(k)}$ is rejected and the trust-region radius is shrunk. 

Convergence of the iterative process is declared when a suitable
criticality measure is sufficiently small. We consider the
measure 
\begin{equation}\label{eq:crit}
  \chi(\mathbf x^{(k)}) = \|P_{\Omega}(\mathbf x^{(k)} - \nabla \f(\mathbf x^{(k)})) - \mathbf x^{(k)})\| 
\end{equation}
where $P_{\Omega}$ is the projection onto the feasible set, and  $\chi(\mathbf x^{(k)})$ is the
norm of the projected gradient onto the box, and reduces to $\norm{\nabla \f(\mathbf x^{(k)})}$
when $\mathbf x^{(k)}$ is in the interior of $\Omega$.
The approach is summarized in the pseudocode of Algorithm \ref{alg:tr} where we assume that the
model functions are minimized ``exactly'' (as it  will be the case in our applications, see Section
\ref{sec:impl}).

\begin{algorithm}
  \begin{algorithmic}
  \Require The initial point $\mathbf x^{(0)} \in \Omega$, 
     an initial trust-region radius $\Delta_0$,
    the constants $\eta_1, \eta_2, \gamma_1, \gamma_2$ such that
  \[
  0 < \eta_1 \le \eta_2 <1 \mbox{ and } 0 < \gamma_1\le \gamma_2 <1,
  \]
  $k_{max}>0$ and $\epsilon>0$.
  
  \State Compute $\f(\mathbf x^{(0)})$
  \For{$k = 1, \ldots, k_{max}$}
      \State Define a model $\fr_k$ in $\mathcal B_k\cap \Omega$.
      \State Compute a step $\mathbf s^{(k)}$ that minimizes the model and 
         such that $\mathbf x^{(k)} + \mathbf s^{(k)} \in \mathcal B_k\cap \Omega$.
      \State Compute $\f(\mathbf x^{(k)}+\mathbf s^{(k)})$ and define the ratio (\ref{eq:ratio}).
      \If{$\rho^{(k)} \ge 1$}
      \State define $\mathbf x^{(k+1)}= \mathbf x^{(k)} + \mathbf s^{(k)}$;
      \Else
      \State define $\mathbf x^{(k+1)}=\mathbf x^{(k)}$.
      \EndIf
      \State Update the trust-region radius setting
             \[
              \Delta_{k+1} \in \left \{ \begin{array}{ll}
                                       \left[\Delta_k, \infty \right) & \mbox{ if } \rho^{(k)} \ge \eta_2 \\
                                       \left [\gamma_2\Delta_k, \Delta_k \right ] & \mbox{ if } \rho^{(k)} \in [\eta_1, \eta_2) \\
                                      \left [\gamma_1\Delta_k, \gamma_2 \Delta_k \right] & \mbox{ if } \rho^{(k)} < \eta_1
                                      \end{array}
                             \right.
             \]

      \State Compute the criticality measure (\ref{eq:crit}).
      \If{ $\chi(\mathbf x^{(k+1)}) \le \epsilon$}
      \State \textbf{return} An approximate first-order critical point $\mathbf x^{(k+1)}$.
      \EndIf
      \EndFor
      \end{algorithmic}  
  \caption{Basic trust-region algorithm}
  \label{alg:tr}
\end{algorithm}

In order to obtain a robust and globally convergent trust-region algorithm 
involving approximate models, the following assumptions should hold \cite{trBible} 
for the objective $\f$ and the model $\fr_k$ functions:
\begin{description}
 \item{(AF.1)} $\f:\IR^\ell \rightarrow \IR$ is twice-continuously differentiable on $\Omega$.
 \item{(AF.2)} The function $\f$ is bounded below for all $\mathbf x \in \Omega$.
 \item{(AF.3)} The second derivatives of $\f$ are uniformly bounded for all $\mathbf x \in \Omega$.
\vskip 8pt
 \item{(AM.1)} For all $k$, $\fr_k$ is twice differentiable on ${\mathcal B}_k$.
 \item{(AM.2)} The values of the objective and the model function coincide at the current iterate, i.e., for
all $k$,
$$\f (\mathbf x^{(k)} ) = \fr_k (\mathbf x^{(k)} ).$$
\item{(AM.3)} The gradients of the objective and the model function coincide at the current iterate,
i.e., for all $k$,
$$\nabla \f (\mathbf x^{(k)} ) = \nabla \fr_k (\mathbf x^{(k)} ).$$
\item{(AM.4)} The second derivatives of $\fr (\mathbf x^{(k)} )$ remain bounded
within the trust-region ${\mathcal B}_k$ for all $k$.
\end{description}

\begin{theorem}\cite[Section 12.2.2]{trBible}
Under Assumptions (AF.1)-(AF.3) and (AM.1)-(AM.4), every limit point of the sequence $\{\mathbf x^{(k)}\}$ generated by Algorithm 
 \ref{alg:tr} is  first-order critical, i.e.
\[
 \lim_{k\rightarrow \infty} \|P_{\Omega}(\mathbf x^{(k)} - \nabla \f(\mathbf x^{(k)})) - \mathbf x^{(k)})\| = 0.
\]
\end{theorem}

\subsection{Enforcing first-order matching} \label{sec:first-order}

Here we use the trust-region framework described in the previous section to solve
problem (\ref{eq:optpb}) with objective function (\ref{eq:obj}). Note
that conditions
(AF.1) -- (AF.3) are automatically satisfied by our definition of objective function.
As discussed in Section~\ref{sec:lanczos}, the local model 
$\tilde \phi_k^{\mathcal R}(\mathbf x)$
that we
build does not necessarily agree at the first order with the original function
$\phi(\mathbf x)$, which apparently breaks assumption (AM.3) --- and only satisfies
(AM.2). However, if the gradient $\nabla \phi_k^{\cal R}(\mathbf x)$ is known at the
expansion point $\mathbf x^{(k)}$, this can be easily fixed by defining
\begin{equation}\label{eq:corr}
 \phi_k^{\cal R} (\mathbf x) := \tilde \phi_k^{\cal R}(\mathbf x)  + (\phi(\mathbf x^{(k)}) -  \tilde \phi_k^{\cal R}(\mathbf x^{(k)})) +
                             (\nabla \phi(\mathbf x^{(k)}) - \nabla \tilde \phi_k^{\cal R}(\mathbf x^{(k)}))^T(\mathbf x-\mathbf x^{(k)}),
\end{equation}                            
see, e.g., \cite{Biegler03,Alex95,giunta00}

\section{Building the reduced model by Lanczos projection}
\label{sec:reducedmodel}

Our aim in this section is to develop a low-dimensional model of the frequencies
as a function of the parameters. The model must be efficient to evaluate, 
and needs to satisfy the requirements for use as a local model in the
trust-region approach described in Section~\ref{sec:trustregion}. 

This task can be viewed as a special case of parametric model order reduction. 
Several
approaches have been developed in this field that allow obtaining low-dimensional
models for accurately approximating the behavior of complex systems. For instance, 
when controlling large-scale dynamical systems, it is possible to employ projection
approaches when the observations and inputs have low-rank properties (see \cite{benner2015survey}
and references therein). 

Our aim is slightly different; we are interested only in
a particular region on the spectrum (the lowest frequencies), but 
we have no particular assumptions on the input (which is typically
induced by the environment). 

In the next sections, we first briefly summarize the standard 
(inverse) Lanczos iteration for matrix pencils, 
mainly to fix the notation, and then we show that we can perform some
modifications to make it parameter dependent. We refer
to \cite{bai2002krylov} for a discussion of the relation between the Lanczos 
process and (nonparametric) model reduction.

\subsection{The Lanczos method}
\label{sec:lanczos}

The Lanczos iteration, for a symmetric
$n \times n$
matrix $A$, can be summarized as follows. 
We construct an orthogonal basis $W_m$
for the subspace $\mathcal K_{m}(A, v) \subseteq \mathbb R^{n}$
defined as: 
\[
  \mathcal K_{m}(A, v) := \mathrm{span}(v, Av, \ldots, A^{m-1} v ). 
\]
Unless a breakdown occurs, this space has dimension $m$,
and the basis $W_m$ is represented as a $n \times m$ orthogonal matrix
(that is, $W_m^T W_m = I_m$). Given the basis, one then computes the projection
$T_m = W_m^T A W_m$. The matrix $T_m$ is
tridiagonal, and its eigenvalues (known as Ritz values) are
good approximations of the extremal eigenvalues of $A$ even when
$m \ll n$. This procedure is appealing for several reasons. 
\begin{enumerate}[(i)]
\item Computation of the space $W_m$ only requires matrix-vector
  multiplications, and therefore it is easy to exploit the sparsity or,
  more generally, the structure of $A$.
\item The orthogonalization of the basis requires only scalar products,
  which can be computed very efficiently if $m$ does not increase
  too much.
\item The algorithm can be carried out iteratively, i.e., 
  the optimal value of $m$ does not need to be known a priori, 
  and $T_m$ can be built incrementally. In our framework, 
  we stop the iteration when the error on the computed 
  eigenvalues is guaranteed to be smaller than a certain
  relative threshold $\tau$ (see \cite{demmel1997applied} for a classical
  reference on the stopping criterion of Lanczos for computing eigenvalues). 
  The actual value of $\tau$ is reported in Section~\ref{sec:examples}. 
\end{enumerate}

When the matrix $A$ arises from a FE
discretization of a PDE involving only local operators
and elements with compact support and few
overlappings, the number of nonzero elements is only $\mathcal O(n)$;
hence, the matrix vector
multiplication can be carried out with linear complexity. Moreover, 
the matrix usual has a band or multiband structure, and using good 
sparse solvers or appropriate preconditioners we can solve linear
systems in $\mathcal O(n)$ time. 
Therefore, the smallest eigenvalue of $A$ can be extracted applying 
Lanczos to $A^{-1}$. 
Since we are interested in modal analysis of structures discretized by FE,
we need
to approximate the smallest eigenvalues of a pencil
$K - \lambda M$, with $K$ and $M$ symmetric
positive definite; 
the same technique to extract the smallest eigenvalues can be applied
(implicitly) to the symmetric matrix $M^{\frac 1 2} K^{-1} M^{\frac 1 2}$,
obtaining: 
\[
  T_m = W_m^T M^{\frac 1 2} K^{-1} M^{\frac 1 2} W_m = U_m^T M K^{-1} M U_m, \qquad
  U_m = M^{-\frac 1 2} W_m. 
\]
It is immediately verifiable that $U_m$ contains a basis
of the same subspace spanned by $W_m$, but this matrix is
$M$-orthogonal (that is, orthogonal with respect to the inner
product defined by $M$), so $U_m^T M U_m = I_m$ \cite{ericsson1980spectral}.

\begin{remark}
  In the above framework, there is no need to explicitly compute the square
  root of the positive definite matrix $M$, which would be a very
  expensive operation. The matrix $T_m$ is implicitly obtained 
  using the matrix-vector product with $M K^{-1} M$ and 
  re-orthogonalization by the scalar-product induced by $M$, 
  computing the basis $U_{m}$ directly. 
\end{remark}

However, using this procedure as a black-box to evaluate the
objective function in an optimization method can easily become 
too onerous: even if the complexity is only linear, the size of the problem can make
this operation unpractical (we often deal with $n \approx 10^6$ or
even more in finite element models).

Therefore, we wish to investigate the use of the Lanczos process
to build a local parametric reduced model for the case when $K$ and
$M$
are not constant matrices, but depend on $\ell$ parameters
$x_1, \ldots, x_\ell$. 

\subsection{Parametric finite element models}

Let us assume that a finite element model
depending on a variable $\mathbf x \in \mathbb R^{\ell}$ is given. These parameters
$\mathbf x$
may encode different properties of the system (our method does not
make any assumption about what they mean --- but only on
their algebraic properties). Most of the time, these will
be some materials' physical characteristics, such as 
Young's modulus, or mass density. In almost all the cases of 
interest, the dependency is \emph{linear} in $\mathbf x$. However, 
for our discussion we just assume it to be smooth, i.e., at least
$\mathcal C^2$.

\subsection{A first-order approximation to the updated Lanczos projection}

In this section we develop a strategy that, given a certain 
initial value $\mathbf x^{(0)}$
for the model parameters, builds a local model to be
used in the trust-region scheme presented in Section~\ref{sec:trustregion}. 

Let us denote with $F(\mathbf x)$ the matrix valued function that,
for a given choice of the parameters $\mathbf x$, returns the
tridiagonal matrix obtained running the Lanczos process
on $M^{\frac 1 2}(\mathbf x) K(\mathbf x)^{-1} M^{\frac 1 2}(\mathbf x)$.

We shall approximate $F(\mathbf x)$ in a neighborhood of $\mathbf x^{(0)}$. 
This approximation calls for two steps: first, we fix
the subspace used for the projection, as described in
Lemma~\ref{lem:f0}, obtaining an approximation $F_0(\mathbf x)$ of $F(\mathbf x)$. 
Then, we show how to further approximate this function in order to make
it cheap to evaluate. Lemma~\ref{lem:objectivef} describes how to
combine these approximations to obtain a 
 first-order correct 
local model  for the trust-region scheme. 

\begin{remark}
    The outcome of the Lanczos process depends on the initial vector
    chosen for the scheme. Here and in the following, we assume that this 
    vector is fixed whenever we keep $\mathbf x^{(0)}$ unchanged. This makes
    the subspace generated by the Lanczos projection unique.
\end{remark}

\begin{lemma} \label{lem:f0}
	Let $F(\mathbf x) = U_m(\mathbf x) M(\mathbf x) K^{-1}(\mathbf x) M(\mathbf x) U_m(\mathbf x)$ be the matrix-valued function that 
	associates the point $\mathbf x$ with the 
	Lanczos projection of the pencil $K(\mathbf x) - \lambda M(\mathbf x)$ previously defined. 	
	Then, there exists a matrix-valued function $U_{m,0}(\mathbf x)$ such
	that, in a neighborhood $\mathcal N$ of $\mathbf x^{(0)}$, 
	\begin{enumerate}
        \item For any $\mathbf x \in \mathcal N$ the 
          matrix-valued function $U_{m,0}(\mathbf x)$ is an
          $M(\mathbf x)$-orthogonal
          basis for the column-span of $U_m(\mathbf x^{(0)})$.
        \item The function $F_0(\mathbf x)$ defined as
        \[
          F_0(\mathbf x) := U_{m,0}(\mathbf x)^T M(\mathbf x) K(\mathbf x)^{-1} M(\mathbf x) U_{m,0}(\mathbf x)
        \]
        is  $\mathcal C^2(\mathcal N)$ 
        and $F_0(\mathbf x^{(0)}) = F(\mathbf x^{(0)})$, that 
        is it is a zeroth-order approximation of $F(\mathbf x)$. 
	\end{enumerate}
\end{lemma}

\begin{proof}
    By construction, the basis generated by the Lanczos method
    at $\mathbf x^{(0)}$ returns an $M(\mathbf x^{(0)})$-orthogonal
    basis $U_{m}(\mathbf x^{(0)})$. 
    We have 
    \[
      U_m(\mathbf x^{(0)})^T M(\mathbf x) U_m(\mathbf x^{(0)}) =  
        Z(\mathbf x), \qquad 
        Z(\mathbf x^{(0)})  = I, 
    \]
    and 
    since $M(\mathbf x)$ is of class $\mathcal C^2$, the same can be said of 
    the symmetric matrix $Z(\mathbf x)$. Therefore, we can define 
    $U_{m,0}(\mathbf x) := U_m(\mathbf x^{(0)}) Z(\mathbf x)^{-\frac 12}$. 
    For $\mathbf x$ close enough to $\mathbf x^{(0)}$, the value 
    of $Z(\mathbf x)$ is bounded away from singular matrices (recall that 
    $Z(\mathbf x^{(0)}) = I$), so the inverse square-root is locally analytic
    and therefore $U_{m,0}(\mathbf x)$ is at least $\mathcal C^2(\mathcal N)$
    in a neighborhood $\mathcal N$. Direct substitution yields $U_{m,0}^T(\mathbf x) M(\mathbf x) U_{m,0}(\mathbf x) = I$ 
    and $F_0(\mathbf x^{(0)}) = F(\mathbf x^{(0)})$, 
    as requested. 
\end{proof}

The idea behind the approximation of Lemma~\ref{lem:f0} is 
to obtain the eigenvalues of 
$F_{0}(\mathbf x)$ by a subspace projection
method, where the subspace is approximated by choosing the one
obtained by running Lanczos 
for the parameters at $\mathbf x^{(0)}$ instead of the one 
at $\mathbf x$. If the two values are close, this subspace
is still ``good enough'' to provide accurate approximations of the 
spectrum. In fact a similar strategy underlies
several methods known as \emph{Krylov subspace recycling}
\cite{parks2006recycling,soodhalter2014krylov}. These techniques find
applications in solving sequences of shifted linear systems, which is 
a step required, for instance, in model reduction algorithms. 

The definition of $F_0(\mathbf x)$, as is, does not make it 
particularly easier to evaluate with respect to $F(\mathbf x)$. 
In Section~\ref{sec:firstorder}, we will show how to approximate
$F_0(\mathbf x)$ at the first-order in a way that makes its 
computation very efficient. The next result justifies this approach.

\begin{lemma} \label{lem:firstorderf}
    Let $\hat F_0(\mathbf x)$ be any locally $\mathcal C^2$
    first-order symmetric positive definite
    approximation of $F_0(\mathbf x)$, and let $\hat\Lambda_s(\cdot)$ be 
    the function that associates a symmetric 
    matrix with the inverse of its largest $s$ eigenvalues. If the eigenvalues 
    at $\mathbf x^{(0)}$ of $F(\mathbf x^{(0)})$ are all 
    distinct then
    \[
      \hat\Lambda_s(F(\mathbf x)) = \hat\Lambda_s(\hat F_0(\mathbf x)) + 
      O(\norm{\mathbf x - \mathbf x^{(0)}}^2), \qquad 
      \mathbf x \in \mathcal N, 
    \]
    where $\mathcal N$ is an appropriate neighborhood of $\mathbf x^{(0)}$. 
\end{lemma}

\begin{proof}
    When the eigenvalues of a symmetric matrix are all distinct, they
    locally depend analytically on the entries of the matrix (see, for instance, 
    \cite{kato2013perturbation,sun1990multiple}). Since the
    matrices involved are all positive definite, the same holds for
    their inverses $\lambda_{i}^{-1}$. 
     Therefore, the result follows by composing
    $\hat\Lambda_s(\cdot)$ with the functions $F(\mathbf x)$ and 
    $\hat F_0(\mathbf x)$. 
\end{proof}

\begin{remark}
    In the event two or more eigenvalues match the situation is slightly
    more problematic, 
    since the dependency of the eigenvalues is only continuous in general, 
    and although partial derivatives exist, 
    a higher-order smoothness in the eigenvalue functions cannot be claimed \cite{sun1990multiple}. 
    For simplicity, we restrict our attention to the case where the eigenmodes are
    separated enough not to collide while optimizing the
    parameters. This is the
    case in all practical applications analyzed in Section~\ref{sec:examples}. 
    However, a greater generality could be treated by replacing
    eigenvectors with deflating subspaces, at the price of a 
    more involved treatment. 
\end{remark}

\subsubsection{First-order expansion of the projection}
\label{sec:firstorder}

In view of Lemma~\ref{lem:firstorderf}, in this Section we aim
at constructing an approximation to $F_0(\mathbf x)$ that matches
at the first-order, and that is numerically undemanding to evaluate. More
precisely, we aim at complexity $\mathcal O(m^3)$, 
where $m$ is the dimension of the Krylov subspace. Let us write
the value of $K(\mathbf x)$ and $M(\mathbf x)$ as small perturbations
of their values at $\mathbf x^{(0)}$:
\[
  K(\mathbf x) = K(\mathbf x^{(0)}) + \delta K(\mathbf x), \qquad
  M(\mathbf x) = M(\mathbf x^{(0)}) + \delta M(\mathbf x). 
\]
We can expand the above expressions to the
first-order, which yields
\begin{equation}
  \label{eq:secondorderK}
  K(\mathbf x) = K(\mathbf x^{(0)}) +
  \sum_{j = 1}^{\ell} (x_j - x^{(0)}_j) \frac{\partial K(\mathbf x^{(0)})}{\partial x_j}
  + O(\norm{\mathbf x - \mathbf x^{(0)}}^2),
\end{equation}
and the analogous formula for $M(\mathbf x)$. Let us assume that we are
given an $M(\mathbf x)$-orthogonal basis $U_{m,0}(\mathbf x)$
spanning the same subspace of $U_m(\mathbf x^{(0)})$ that we have computed
with the Krylov projection method at $\mathbf x^{(0)}$. In this case,
 we can compute the new
projected counterpart of $M^{\frac 12}(\mathbf x) K(\mathbf x)^{-1} M^{\frac 12}(\mathbf x)$ as follows:
\begin{equation}
  \label{eq:projection}
  F_{0}(\mathbf x) = U_{m,0}^T(\mathbf x) M(\mathbf x) K(\mathbf x)^{-1} M(\mathbf x) U_{m,0}(\mathbf x).  
\end{equation}
In order to derive a first-order expansion for the above formula we need
first-order expansions for all the terms involved. However, we still do
not have such an expression for the inverse of $K(\mathbf x)$. 
For $\mathbf x$ sufficiently close to $\mathbf x^{(0)}$ 
we can write the Neumann expansion:
\[
  K(\mathbf x)^{-1} = (I + K(\mathbf x^{(0)})^{-1} \delta K(\mathbf x))^{-1} K(\mathbf x^{(0)})^{-1} =
  \sum_{j = 0}^{\infty} (-1)^j \left[ K(\mathbf x^{(0)})^{-1} \delta K(\mathbf x) \right]^j K(\mathbf x^{(0)})^{-1}. 
\]
Relying on Equation~(\ref{eq:secondorderK}), and dropping the second-order
terms yields
\begin{align*}
  K(\mathbf x)^{-1} = K(\mathbf x^{(0)})^{-1} &-
  \sum_{j = 1}^\ell (x_j - x_j^{(0)}) K(\mathbf x^{(0)})^{-1} \frac{\partial K(\mathbf x^{(0)})}{\partial x_j} K(\mathbf x^{(0)})^{-1}
+ O(\norm{\mathbf x - \mathbf x^{(0)}}^2), 
\end{align*}
which in turn can be rewritten in the following form:
\begin{equation}
  \label{eq:kexpansion}
K(\mathbf x)^{-1} = K(\mathbf x^{(0)})^{-1} \left( I - \sum_{j = 1}^\ell (x_j - x_j^{(0)}) \frac{\partial K(\mathbf x^{(0)})}{\partial x_j} K(\mathbf x^{(0)})^{-1} \right) + O(\norm{\mathbf x - \mathbf x^{(0)}}^2). 
\end{equation}
Equation~(\ref{eq:kexpansion}) has the desirable property that all the 
inverses of $K(\mathbf x)$ appear evaluated at $\mathbf x^{(0)}$,
and the dependency on $\mathbf x$ is linear.

We now have all the ingredients to derive a cheap formula for computing $F_{0}(\mathbf x)$, 
by substituting
the expansions in Equation~(\ref{eq:projection}). This leads to:
\begin{equation} \label{eq:projected-expanded}
  \begin{split}
  F_{0}(\mathbf x) &= U_{m,0}^T(\mathbf x) M(\mathbf x^{(0)}) K(\mathbf x^{(0)})^{-1} M(\mathbf x^{(0)}) U_{m,0}(\mathbf x) \\
                   &+ \sum_{j = 1}^\ell (x_j - x_j^{(0)}) U_{m,0}(\mathbf x)^T G_j(\mathbf x^{(0)}) U_{m,0}(\mathbf x) + O(\norm{\mathbf x - \mathbf x^{(0)}}^2),
\end{split}
\end{equation}
where $G_j(\mathbf x^{(0)})$ is obtained by grouping together all the first-order contributions, that is, 
\begin{align*}
  G_j(\mathbf x^{(0)}) &= \frac{\partial M(\mathbf x^{(0)})}{\partial x_j} K(\mathbf x^{(0)})^{-1} M(\mathbf x^{(0)}) 
  - M(\mathbf x^{(0)}) K(\mathbf x^{(0)})^{-1}\frac{\partial K(\mathbf x^{(0)})}{\partial x_j} K(\mathbf x^{(0)})^{-1} M(\mathbf x^{(0)}) \\
  &+ M(\mathbf x^{(0)}) K(\mathbf x^{(0)})^{-1} \frac{\partial M(\mathbf x^{(0)})}{\partial x_j}. 
\end{align*}
The terms $G_j(\mathbf x^{(0)})$ are 
large matrices, and computing them explicitly would be
unfeasible. However,
by observing Equation~(\ref{eq:projected-expanded}) closely, 
we immediately notice that
we do not need the complete matrices, but just their projected counterparts
$\hat G_j(\mathbf x) := U_{m,0}(\mathbf x)^T G_j(\mathbf x^{(0)}) U_{m,0}(\mathbf x)$,
which are much smaller.
Moreover, in the definition of $G_j(\mathbf x^{(0)})$ 
the only matrix
that appears as an inverse is $K(\mathbf x^{(0)})$. In our Lanczos implementation
we use a direct sparse solver so, after the projection, we already have
a sparse Cholesky factorization of $K(\mathbf x^{(0)})$, and we compute
$\hat G_j(\mathbf x^{(0)})$ 
at the cost of some extra back-substitutions and
sparse matvec products (which have a comparable cost).
Similar savings could be obtained recycling the preconditioner
when using an iterative method. 

However, the same computational cost would be needed to calculate $\hat G_j$ 
at $\mathbf x \neq \mathbf x^{(0)}$. We would like to avoid this extra
computation. To this end, observe that according to the proof 
of Lemma~\ref{lem:f0}, we have $U_{m,0}(\mathbf x) = U_m(\mathbf x^{(0)}) Z(\mathbf x)^{- \frac 12}$. 
Therefore, we can write
the formulas to compute $\hat G_j(\mathbf x)$
as follows: 
\[
  \hat G_j(\mathbf x) = Z(\mathbf x)^{-\frac 12} U_m(\mathbf x^{(0)})^T G_j(\mathbf x^{(0)}) U_m(\mathbf x^{(0)}) Z(\mathbf x)^{-\frac 12}, \qquad
\]
which shifts the problem to that of computing
$Z(\mathbf x) = U_m(\mathbf x^{(0)})^T M(\mathbf x) U_m(\mathbf x^{(0)})$. As in the previous equations, we can obtain
a first-order approximation of $Z(\mathbf x)$ at any $\mathbf x$
with a first-order
truncation by precomputing some small projected matrices at $\mathbf x^{(0)}$:
\begin{align*}
  \hat Z(\mathbf x) &:= U_m(\mathbf x^{(0)})^T \left(
    M(\mathbf x^{(0)}) + \sum_{j = 1}^\ell (x_j - x_j^{(0)}) \frac{\partial M(\mathbf x^{(0)})}{\partial x_j} \right)
    U_m(\mathbf x^{(0)}) \\
    &= I + \sum_{j}^\ell (x_j - x_j^{(0)}) U_m(\mathbf x^{(0)}) ^T\frac{\partial M(\mathbf x^{(0)})}{\partial x_j} U_m(\mathbf x^{(0)}).
 \end{align*}

 We can now derive a formula for an approximation $F_{m,0}(\mathbf x)$ 
 of $F_0(\mathbf x)$
by combining all these considerations:
 \begin{equation}\label{eq:Tapp}
  F_{m,0}(\mathbf x) = \hat Z(\mathbf x)^T T_m(\mathbf x_0) \hat Z(\mathbf x) + \sum_{j = 1}^\ell (x_j - x_j^{(0)}) \hat G_j(\mathbf x) 
 \end{equation}
 According the the previous remarks, $\hat Z(\mathbf x)$, 
 and $\hat G_{j}(\mathbf x)$, 
 can be computed cheaply in 
 $\mathcal O(m^3)$ flops, so with a computational cost that is 
 \emph{independent} of the size of the original problem. 
In view of this analysis, we arrive at the following result. 

\begin{lemma} \label{lem:stage2}
    The matrix-valued function $F_{m,0}(\mathbf x)$ 
    defined in \eqref{eq:Tapp} is a first-order approximation
    of $F_0(\mathbf x)$ at $\mathbf x = \mathbf x^{(0)}$. More precisely, 
    there exists a neighborhood $\mathcal N$ containing $\mathbf x^{(0)}$
    such that
    \[
      F_{0}(\mathbf x) = F_{m,0}(\mathbf x) + O(\norm{\mathbf x - \mathbf x^{(0)}}^2), \qquad 
      \mathbf x \in \mathcal N.
    \]
\end{lemma}

\subsubsection{Restarted Lanczos and other subspace iterations}

At this stage, at no time we have used
the property that $W_{m}(\mathbf x)$ spans a Krylov subspace
generated by $K(\mathbf x)^{-1}$. In fact, this is not a requirement at all. 
Any subspace $W_{m}(\mathbf x)$ generated through a subspace iteration
method will work equally well in this framework. 

In particular, we could rely on well-known large-scale eigenvalue
solver such as ARPACK \cite{lehoucq1998arpack}, which provides a tried-and-tested
implementation of a restarted Lanczos method to compute the smaller
end of the spectrum. This is what we currently use in the NOSA code. 

Other approaches can be considered as well. The only requirement is that 
the projected matrix have the form 
\[
  W_m(\mathbf x)^T M(\mathbf x)^{\frac 12} K(\mathbf x)^{-1} M(\mathbf x)^{\frac 12} W_m(\mathbf x),
\] which
in turn implies that we are interested in the lowest end of the spectrum. 
For this reason, and in order to keep the explanation simple and self-contained, 
we decided to consider only the Lanczos method herein. Even if we 
never mention restarted variants explicitly, their use does not require any
change in the proposed strategy. 

\subsubsection{Computing derivatives} \label{sec:derivatives}

In order to apply Algorithm \ref{alg:tr} and, in particular, to perform the first-order
correction  (\ref{eq:corr}), we need to be able
to evaluate $\nabla\phi(\mathbf x^{(k)})$ and $\nabla\phi_k^{\mathcal R}(\mathbf x^{(k)})$. 
In both cases, the problem can be rephrased in terms of the computation of 
a derivative of the eigenvalues of a pencil $K - \lambda M$, in view of 
\begin{equation} \label{eq:scalar-product-eigs}
  \frac{\partial}{\partial x_j} \norm*{\frac{\sqrt{\Lambda_s(K(\mathbf x), M(\mathbf x))}}{2 \pi} - \mathbf f}_{\mathbf w, 2}^2 =
      \left(
       \frac{\sqrt{\Lambda_s(K(\mathbf x), M(\mathbf x))}}{2 \pi} - \mathbf f
       \right)^T  
    D_{\mathbf w}^2 
      \frac{ 
      	  \frac{\partial}{\partial x_j} \Lambda_s(K(\mathbf x), M(\mathbf x))
      	}{
      	  2 \pi \sqrt{\Lambda_s(K(\mathbf x), M(\mathbf x))}
      	},
\end{equation}
where $D_{\mathbf w} = \mathrm{diag}(w_1, \ldots, w_s)$, as in 
\eqref{eq:obj}.
Classical perturbation theory for eigenvalues yields the following, for $i = 1, \ldots, s$:
\[
  \left[ \frac{\partial}{\partial x_j} \Lambda_s(K(\mathbf x), M(\mathbf x)) \right]_i = 
    \left(v_i^T \left( \frac{\partial K(\mathbf x)}{\partial x_j} - \lambda_i \frac{\partial M(\mathbf x)}{\partial x_j} \right) v_i\right) \cdot \left( 
      v_i^T M(\mathbf x) \, v_i
    \right)^{-1}, 
\]
where $v_i$ is an eigenvector relative to the eigenvalues $\lambda_i$. 

\subsection{Constructing a local model for the trust-region scheme}

Now that we have an approximation $F_{m,0}(\mathbf x)$ of $F(\mathbf x)$, 
we can obtain a local model for the trust-region scheme by replacing 
$\hat\Lambda_s(F(\mathbf x))$ with $\hat\Lambda_s(F_{m,0}(\mathbf x))$. 

In view of Lemma~\ref{lem:f0}, the approximation of $F(\mathbf x)$ 
with $F_{m,0}(\mathbf x)$ is correct only up to the zeroth order, and 
the same 
holds when composing it with $\hat\Lambda_s(\cdot)$, and so
for $\f(\mathbf x)$ and $\fr(\mathbf x)$ as well. However, 
using the strategy described in Section~\ref{sec:firstorder}, 
we can modify $\fr(\mathbf x)$ a posteriori to match at the first-order. 

Since both
$\f(\mathbf x)$ and $\fr(\mathbf x)$ are of the form described in equation
\eqref{eq:scalar-product-eigs}, and in both cases the eigenvectors are available
while computing the eigenvalues, the same formula can be used directly. 

\begin{lemma} \label{lem:objectivef}
    Let $\fr(\mathbf x)$ be the function defined as follows
    \begin{equation}\label{eq:modtr}
      \fr(\mathbf x) = \tilde\fr(\mathbf x) + (\nabla \f(\mathbf x^{(0)}) - 
      \nabla \tilde \fr(\mathbf x^{(0)}))^T (\mathbf x - \mathbf x^{(0)}), \quad 
      \tilde \fr(\mathbf x) = \norm*{\frac{\sqrt{\hat\Lambda_s(F_{m,0}(\mathbf x))}}{2\pi} - \mathbf f}_{\mathbf w, 2}^2.      
    \end{equation}
    Then, $\fr(\mathbf x)$ is a first-order approximation of $\f(\mathbf x)$ 
    at $\mathbf x = \mathbf x^{(0)}$. 
\end{lemma}

\begin{proof}
    In view of Lemma~\ref{lem:f0}, \ref{lem:firstorderf}, \ref{lem:stage2}
     we have $\f(\mathbf x^{(0)}) = \tilde \fr(\mathbf x^{(0)})$. Given this condition, the
    first-order matching is ensured by the gradient correction
    done as described with Section~\ref{sec:firstorder}. 
\end{proof}

\begin{remark}
    Although the formulas for $\nabla \f(\mathbf x)$ and $\nabla \fr(\mathbf x)$ are not reported explicitly in Lemma~\ref{lem:objectivef}, they 
    are directly implementable in view of Equation~\eqref{eq:scalar-product-eigs}, and require almost 
    no computational effort when the eigenvectors are obtained through
    the Lanczos process. 
\end{remark}


\section{Numerical experiments and real examples}
\label{sec:examples}

\newcommand{\mpa}{\mathrm{MPa}}
\newcommand{\kgm}{\mathrm{\frac{kg}{m^3}}}

In this section we verify the performance of the proposed
approach on both artificial and
real examples. To this end, we will compare the number of iterations needed
to achieve convergence, which is reported in terms of outer iterations, that is
to say, the number of reduced models
computed in order to complete the procedure.

The number of iterations of the inner optimizer, used within a single trust-region,
are ignored because they are largely irrelevant in determining the
final computational cost. Typical tests show that the total time spent optimizing
the models within the trust-region is less than $1\%$ of the total running
time of the algorithm.

\subsection{Some implementation details}\label{sec:impl} 

The tests were run on a computer with an Intel Core i7-920 CPU running at 2.67 GHz,
with 18 GB of RAM clocked
at 1066 MHz. We used the single-core version of MUMPS 4.10, and
the Intel MKL BLAS shipped with MATLAB R2017b.

In all experiments the tolerance $\tau$ for the accuracy of the
Lanczos method was set to $10^{-5}$ while the trust-region procedure was
 stopped when
the norm of the projected gradient norm fell below  $10^{-4}$
(i.e. $\epsilon = 10^{-4}$ in Algorithm \ref{alg:tr}). All the other parameters
in Algorithm \ref{alg:tr} were set to standard values as suggested in
\cite[Chapter 17]{trBible}. The trust-region radius update follows
\cite[Chapter 17, page 782]{trBible} as well. Also, in Algorithm \ref{alg:tr}
the trust-region step $\bf s^{(k)}$ is computed ``exactly'', that is, we
 minimize the reduced model (\ref{eq:modtr}) in ${\cal B}_k \cap \Omega$
 to high accuracy. This is a reasonable choice due to problem low dimension of the
 problem.
To this end, we used
the function $\texttt{fmincon}$ included in the \matlab's Optimization
toolbox, setting the built-in \texttt{sqp} solver and using the
default parameter settings.

Furthermore, the parameter space is always preliminary
scaled, so that the initial point $\mathbf x^{(0)}$ is the vector of all ones. This ensures
that checking norm-wise conditions on the eigenvalues and on the gradient
yields relative accuracy on the parameters, independently of their scaling.

The weight vector $\mathbf w$ is always chosen as $w_i = f_i^{-1}$, which ensures
relative accuracy on the recovered frequency, except where otherwise stated. 
In particular, in the clock tower example, we emphasize the consequences
of a different choice of $\mathbf w$.

\subsection{Arch on piers}

In this section we consider a simple example for demonstration purposes. We built
a 2D discretization of the arch on piers, shown in Figure~\ref{fig:portal}.
The arch spans $4$ meters, and rests on two $4$-meters-high lateral piers. The structure is modeled by means of $336$ finite elements, and clamped
at the base of the two piers. This corresponds to $851$ total degrees of freedom
in the structure. We have used plane strain $4$-node quadrilateral elements. 

The
three materials composing the structure are depicted with different colors
in Figure~\ref{fig:portal}.
Material $1$ is used for the arch (green in the figure), 
and materials $2$ and $3$ for the piers (red and blue, respectively).
Poisson's ratio $\nu$ is set to $0.2$ for all the materials, and assumed
to be known a priori.
\begin{figure}
    \centering
    \includegraphics[height=.35\textheight]{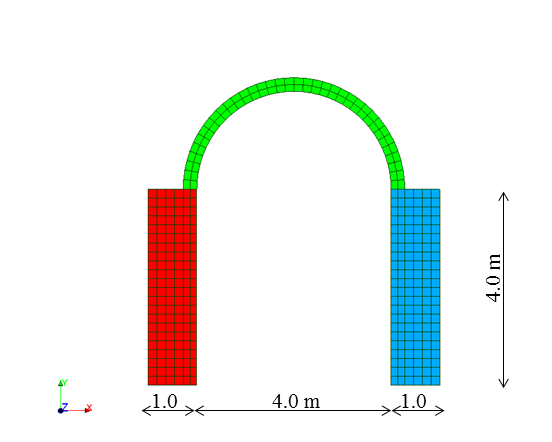}
    \caption{The arch on piers modeled in
        the experiments. Each differently colored region 
        corresponds to a different material. }
    \label{fig:portal}
\end{figure}

\begin{figure} \centering
    \begin{tikzpicture}
    \begin{semilogyaxis}[width=.45\linewidth, height = .35\textheight, ymax = 5,
    ylabel = Obj. function / Proj. gradient, xlabel = Iterations, legend pos = south west,
    ymin = 1e-13]
    \addplot table {dat/convergence_std.dat};
    \addplot table[x index = 0, y index = 7] {dat/convergence_std.dat};
    \addplot[dashed, domain = 0:4] {1e-4};
    \legend{Obj. function, Proj. gradient, Tolerance ($10^{-4}$)}
    \end{semilogyaxis}
    \end{tikzpicture}~~~~~~~~
    \begin{tikzpicture}
    \begin{axis}[width=.45\linewidth, height = .35\textheight,
    ylabel = Frequencies (Hz), xlabel = Iterations, legend pos = north west]
    \addplot table[x index = 0, y index = 2] {dat/convergence_std.dat};
    \addplot[domain = 0 : 4, dashed]{ 9.575048987168955e+00 };
    \addplot table[x index = 0, y index = 3] {dat/convergence_std.dat};
    \addplot[domain = 0 : 4, dashed]{ 1.487136204507410e+01 };
    \addplot table[x index = 0, y index = 4] {dat/convergence_std.dat};
    \addplot[domain = 0 : 4, dashed]{ 2.317329261256074e+01 };
    \addplot table[x index = 0, y index = 5] {dat/convergence_std.dat};
    \addplot[domain = 0 : 4, dashed]{ 3.917029465848877e+01 };
    \addplot table[x index = 0, y index = 6] {dat/convergence_std.dat};
    \addplot[domain = 0 : 4, dashed]{ 6.283760720867433e+01 };
    \end{axis}
    \end{tikzpicture}
    \caption{The left graph shows the convergence of the objective function
        to the minimum for each new reduced model generated in the
        optimization procedure. The dashed line is the tolerance
        set for the projected gradient in the optimization scheme.
        The right graph plots the convergence of the
        frequencies during the process. }
    \label{fig:portal-convergence-near}
\end{figure}
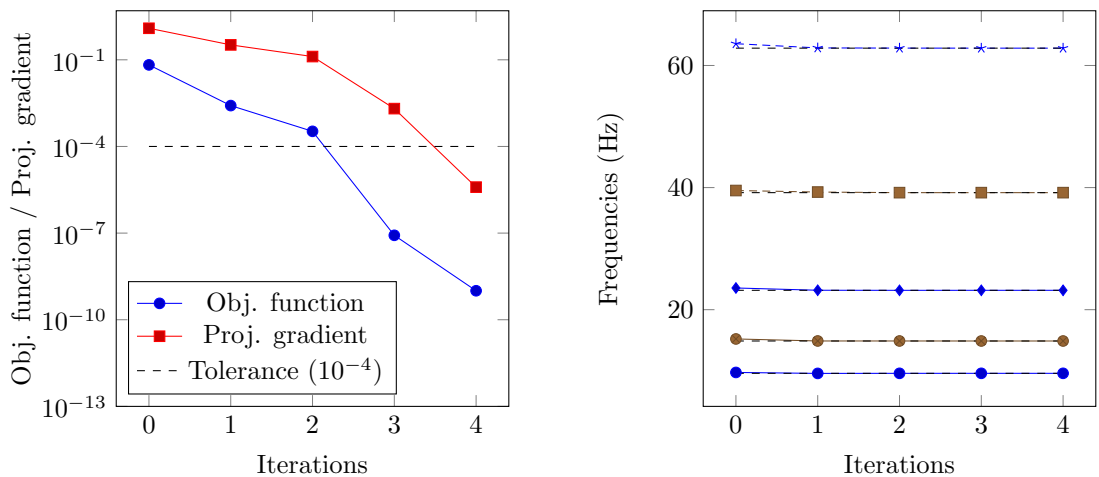

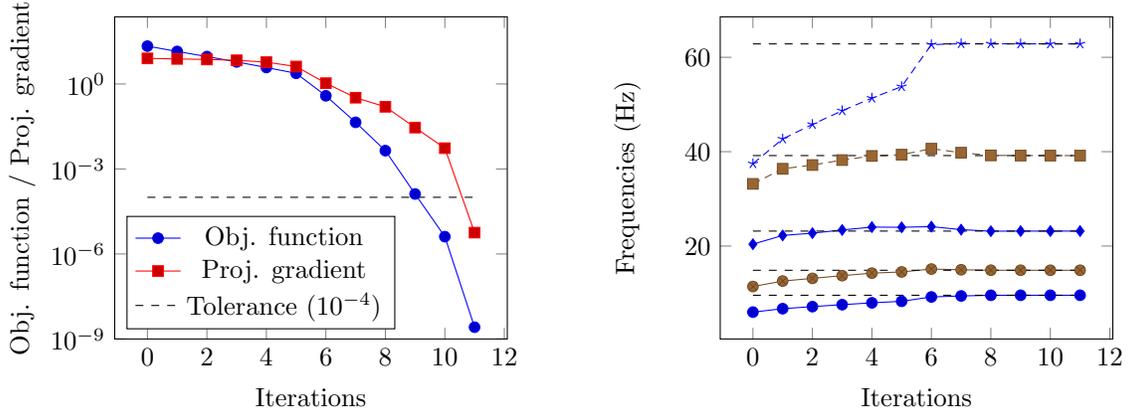
\begin{figure} \centering
    \begin{tikzpicture}
        \begin{semilogyaxis}[width=.45\linewidth, height = .3\textheight,
        ylabel = Obj. function / Proj. gradient, xlabel = Iterations,
        legend pos = south west, ymin = 1e-9]
             \addplot table {dat/convergence_far.dat};
             \addplot table[x index = 0, y index = 7] {dat/convergence_far.dat};
             \addplot[dashed, domain = 0 : 11] {1e-4};
             \legend{Obj. function, Proj. gradient, Tolerance ($10^{-4}$)}
        \end{semilogyaxis}
    \end{tikzpicture}~~~~~~~~
    \begin{tikzpicture}
        \begin{axis}[width=.45\linewidth, height = .3\textheight,
          ylabel = Frequencies (Hz), xlabel = Iterations]
            \addplot table[x index = 0, y index = 2] {dat/convergence_far.dat};
            \addplot[domain = 0 : 11, dashed]{ 9.575048987168955e+00 };
            \addplot table[x index = 0, y index = 3] {dat/convergence_far.dat};
            \addplot[domain = 0 : 11, dashed]{ 1.487136204507410e+01 };
            \addplot table[x index = 0, y index = 4] {dat/convergence_far.dat};
            \addplot[domain = 0 : 11, dashed]{ 2.317329261256074e+01 };
            \addplot table[x index = 0, y index = 5] {dat/convergence_far.dat};
            \addplot[domain = 0 : 11, dashed]{ 3.917029465848877e+01 };
            \addplot table[x index = 0, y index = 6] {dat/convergence_far.dat};
            \addplot[domain = 0 : 11, dashed]{ 6.283760720867433e+01 };
        \end{axis}
    \end{tikzpicture}
    \caption{The left graph shows the convergence of the objective function
        to the minimum for each new reduced model generated in the
        optimization procedure. The dashed line is the tolerance
        set for the projected gradient in the optimization scheme.
        The right graph plots the convergence of the
        frequencies during the process. In this case, the initial points
        were chosen quite far from the original ones, selecting
        $E_2 = 2000\ \mpa, \rho_2 = 1100\ \mathrm{kg\cdot m^{-3}},
          E_3 = 1100\ \mpa$.}
        \label{fig:portal-convergence-far}
\end{figure}

For the three materials we assume
the following values:
\begin{align*}
  E_1 &= 3250\ \mathrm{MPa} & E_2 &= 5000\ \mathrm{MPa} & E_3 &= 4800\ \textrm{MPa} \\
  \rho_1 &= 1800\ \mathrm{\frac{kg}{m^{3}}} & \rho_2 &= 2200\ \mathrm{\frac{kg}{m^{3}}} &
  \rho_3 &= 2100\ \mathrm{\frac{kg}{m^{3}}},
\end{align*}
where the index $j \in \{ 1, 2, 3 \}$ indicates the material under consideration.
We computed the $5$ leading frequencies by running the NOSA-ITACA code  (with
these fixed parameters), to obtain:
\[
  \mathbf f \approx \begin{bmatrix}
    9.575 & 14.87 & 23.17 & 39.17 & 62.84
  \end{bmatrix}.
\]
Our aim is to validate the optimization method by recovering
the original parameters matching
these frequencies, assuming $E_2, \rho_2$, and $E_3$ are unknown.
We consider the following bounds of
realistic parameters:
\begin{align*}
  1000\ \mathrm{MPa} \leq &E_2 \leq 9000\ \mathrm{MPa}, &
  1000\ \mathrm{\frac{kg}{m^{3}}} \leq &\rho_2 \leq 3000\ \mathrm{\frac{kg}{m^{3}}}, &
  1000\ \mathrm{MPa} &\leq E_3 \leq 9000\ \mathrm{MPa},
\end{align*}
which define the corresponding box $\Omega$ in (\ref{eq:Omega})

In the default implementation, the starting points have been chosen the
midpoint of the intervals, which are quite good estimates of the true
values. Convergence in this case is displayed in Figure~\ref{fig:portal-convergence-near}. The initial
points are sufficiently close to the correct ones for the frequencies
to almost match already, and the distance to the optimum during
convergence is not easily distinguishable.  However, the left plot in Figure~\ref{fig:portal-convergence-near}, which shows the convergence
of the objective function in log-scale, clearly demonstrates that we are recovering
the right solution.

To test the robustness of the approach, we modified the starting
points to be relatively far from the correct ones by making the following
choices:
\begin{align*}
     E_2^{(0)} &= 2000\ \mathrm{MPa}, & \rho_2^{(0)} &= 1100\ \mathrm{\frac{kg}{m^{3}}}, &
     E_3^{(0)} &= 1100\ \mathrm{MPa}.
\end{align*}
Convergence of the method is displayed in the graphs in Figure~\ref{fig:portal-convergence-far}, which clearly shows that the
initial frequencies are quite far from the correct ones; yet the method
reaches convergence using only $12$ reduced models. Moreover, a rough estimate (though probably
sufficient for most engineering purposes) is already obtained
after about 6 steps. Moreover, the Young's modulus and the mass density of the
initial model are recovered exactly up to $5$ digits in this example. 

Another important feature for a method that must deal with experimentally
measured frequencies is its robustness when the input is subject to noise.
In fact, we would expect the frequencies to be accurate only up to a
certain relative threshold.

To simulate the behavior of the method in a predictable environment, we
perturbed the frequency vector $\hat {\mathbf f} = \mathbf f + \delta \mathbf f$, by imposing
 $|\delta f_i| \leq |f_i| \cdot \delta$, with $\delta$ the prescribed
noise level. We have run tests for $\delta$ ranging from $0.01 \%$ to
$100 \%$: the corresponding error in the retrieved frequencies is
plotted in Figure~\ref{fig:noise}.

The error is measured in relative norm, that is we have plotted the infinity
norm of the vector with components $(x_i - \hat x_i) / x_i$, where $x_i$
is the vector with the actual model parameters, and $\hat x_i$
those estimated by the optimization process.

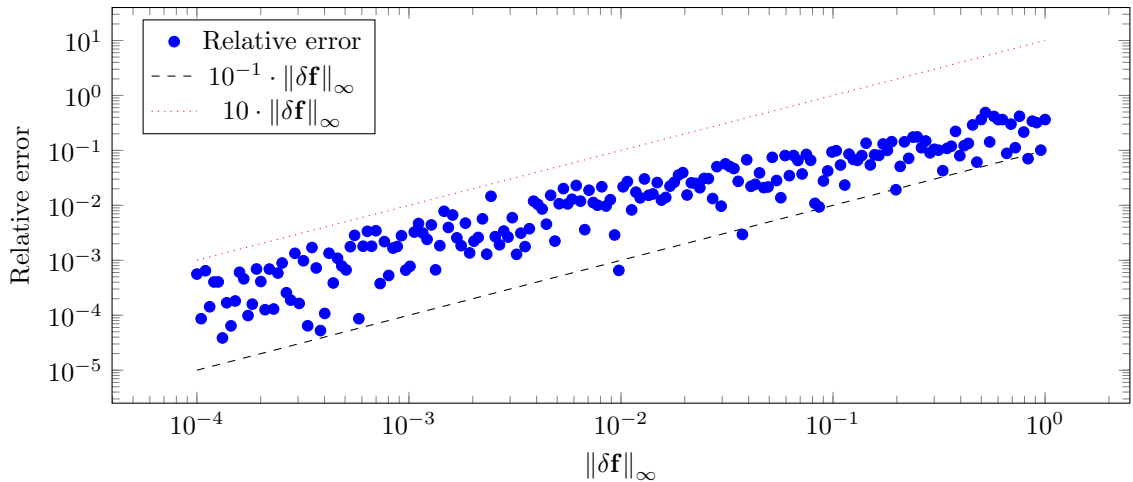
\begin{figure}
    \centering

    \begin{tikzpicture}
    \begin{loglogaxis}[width = \linewidth, legend pos = north west,
    height = .35\textheight, xlabel = $\norm{\delta \mathbf f}_\infty$,
    ylabel = Relative error]
    \addplot[mark=*, blue, only marks] table {dat/arco_noise.dat};
    \addplot[domain = 1e-4 : 1e0, dashed, black] { 1e-1 * x };
    \addplot[domain = 1e-4 : 1e0, dotted, red] { 10 * x };
    \legend{Relative error, $10^{-1} \cdot \norm{\delta \mathbf f}_\infty$, $10 \cdot \norm{\delta \mathbf f}_\infty$};
    \end{loglogaxis}
    \end{tikzpicture}

\caption{This plot shows the
     relation between the noise affecting the frequency measurements 
    and the component-wise relative error on the retrieved optimal
    parameters. The graph shows the linear dependency of the error
    on the estimated parameters and the noise level of the
    input data.}
\label{fig:noise}
\end{figure}

The linear behavior of the relative error with respect to the noise level is
clearly visible. This behavior suggests that,
even if the frequency measurements are contaminated with noise, it
is still possible to retrieve meaningful parameters if enough information
is provided. Note that, in general, correct (or accurate) recovery 
of the parameters might not be possible, and a more detailed study of 
the condition number of this inverse problem will be object of further
in-depth study in 
in future research. 

\subsection{A dome}

Here we provide a more complex example, represented by a dome supported by four
14-meters-high pillars. The dome consists of an octagonal shaped cloister
vault 5 meters high, resting on a drum inscribed on a $10 \times 11$ meters rectangle. We assume the structure to be composed of 4 different materials:
material 1 for the vault (red in Figure~\ref{fig:dome-model}), material 2 for the top of the drum
(blue), material 3 for the lower part of the
drum (green) and material 4 for the pillars (gray).

The finite element model consists of $\num{31052}$ elements and $\num{41245}$ nodes.
It is modeled
using 3D $8$-node hexahedron brick elements, which gives rise to a (projected) model with
$\num{122,853}$ degrees of freedom.

\begin{figure}
    \centering
    \includegraphics[height=.45\textheight]{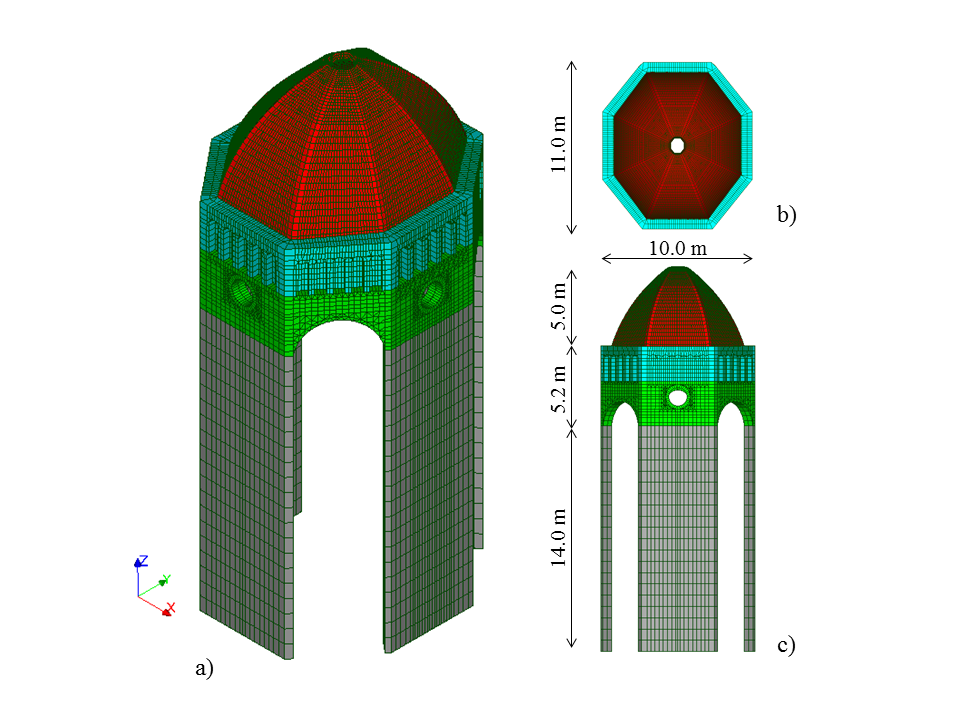}
    \caption{Domed temple}
    \label{fig:dome-model}
\end{figure}

We performed a test similar to that carried out for the arch on piers: 
we computed the frequencies for
a certain choice of parameters, some of which were assumed 
to be unknown, and ran the
optimization scheme trying to recover the original parameters.

In order to obtain the reference frequencies the characteristics 
of the 4 materials (Young's modulus and mass density)
were set as follows:
\begin{align*}
    E_1 &= 3000\ \mpa & \rho_1 &= 1800\ \kgm &
    E_2 &= 4000\ \mpa & \rho_2 &= 2000\ \kgm \\
    E_3 &= 3500\ \mpa & \rho_3 &= 1900\ \kgm &
    E_4 &= 5000\ \mpa & \rho_2 &= 2200\ \kgm \\
    \nu_j &= 0.25, \qquad  j = 1, \ldots, 4.
\end{align*}
The leading $10$ frequencies have been computed using the NOSA-ITACA code. The
result, expressed in Hertz, are the frequencies in the vector:
\[
  \mathbf f \approx \begin{bmatrix}
    2.19 & 2.23 & 3.76 & 3.83 & 4.32 & 4.60 &
    4.72 & 8.26 & 8.30 & 9.21
  \end{bmatrix}.
\]
The optimization code was run setting the bounds as
\[
  2000\ \mpa \leq E_j \leq 6000\ \mpa \qquad
  1600\ \kgm \leq \rho_j \leq 2400\ \kgm \qquad
  j = 1,\ldots, 4,
\]
with the sole exception of $\rho_3$ which was set to the correct value.
This leaves $7$ parameters to be optimized. No
particular starting conditions were specified. In this case, the algorithm
chooses the midpoint of the intervals. The convergence history is reported in Figure~\ref{fig:dome}, which reveals that high accuracy is attained with only
$4$ outer iterations --- with the residual of the objective function
being very small even after two steps.
The weights have been chosen in order to reach relative accuracy\footnote{A detailed discussion of the influence of the weights on the accuracy of the recovered frequencies is presented
    in the next section, where its relevance is discussed for a real example.} (i.e., $w_i = f_i^{-1}$)
on the target frequencies, with the usual tolerance $\epsilon = 10^{-4}$.
\begin{figure}
\centering
\begin{tikzpicture}
\begin{semilogyaxis}[
xtick={0,1,2,3,4},
width=.45\linewidth, height = .35\textheight,
ylabel = Obj. function / Proj. gradient, xlabel = Iterations,
legend pos = north east, ymax = 1e3]
\addplot table[x index = 0, y index = 1] {dat/convergence_massa_relative.dat};
\addplot table[x index = 0, y index = 12] {dat/convergence_massa_relative.dat};
\addplot[dashed, domain = 0:4] {1e-4};
\legend{Obj. function, Proj. gradient, Tolerance ($10^{-4}$)}
\end{semilogyaxis}
\end{tikzpicture}~~~~~~~~
\begin{tikzpicture}
\begin{axis}[width=.45\linewidth, height = .35\textheight,
ylabel = Frequencies (Hz), xlabel = Iterations]
\addplot table[x index = 0, y index = 2] {dat/convergence_massa_relative.dat};
\addplot[domain = 0 : 4, dashed]{ 2.19 };
\addplot table[x index = 0, y index = 3] {dat/convergence_massa_relative.dat};
\addplot[domain = 0 : 4, dashed]{ 2.23 };
\addplot table[x index = 0, y index = 4] {dat/convergence_massa_relative.dat};
\addplot[domain = 0 : 4, dashed]{ 3.76 };
\addplot table[x index = 0, y index = 5] {dat/convergence_massa_relative.dat};
\addplot[domain = 0 : 4, dashed]{ 3.83 };
\addplot table[x index = 0, y index = 6] {dat/convergence_massa_relative.dat};
\addplot[domain = 0 : 4, dashed]{ 4.32 };
\addplot table[x index = 0, y index = 7] {dat/convergence_massa_relative.dat};
\addplot[domain = 0 : 4, dashed]{ 4.6 };
\addplot table[x index = 0, y index = 8] {dat/convergence_massa_relative.dat};
\addplot[domain = 0 : 3, dashed]{ 4.72 };
\addplot table[x index = 0, y index = 9] {dat/convergence_massa_relative.dat};
\addplot[domain = 0 : 3, dashed]{ 8.26 };
\addplot table[x index = 0, y index = 10] {dat/convergence_massa_relative.dat};
\addplot[domain = 0 : 3, dashed]{ 8.30 };
\addplot table[x index = 0, y index = 11] {dat/convergence_massa_relative.dat};
\addplot[domain = 0 : 3, dashed]{ 9.21 };
\end{axis}
\end{tikzpicture}

    \caption{Convergence history for the dome model, matching 
    	$10$ frequencies with $7$ free parameters. The left plot shows the convergence of the objective function
        and the norm of the projected gradient, and in the right plot the convergence
        of frequencies to the exact ones is reported. }
    \label{fig:dome}
\end{figure}
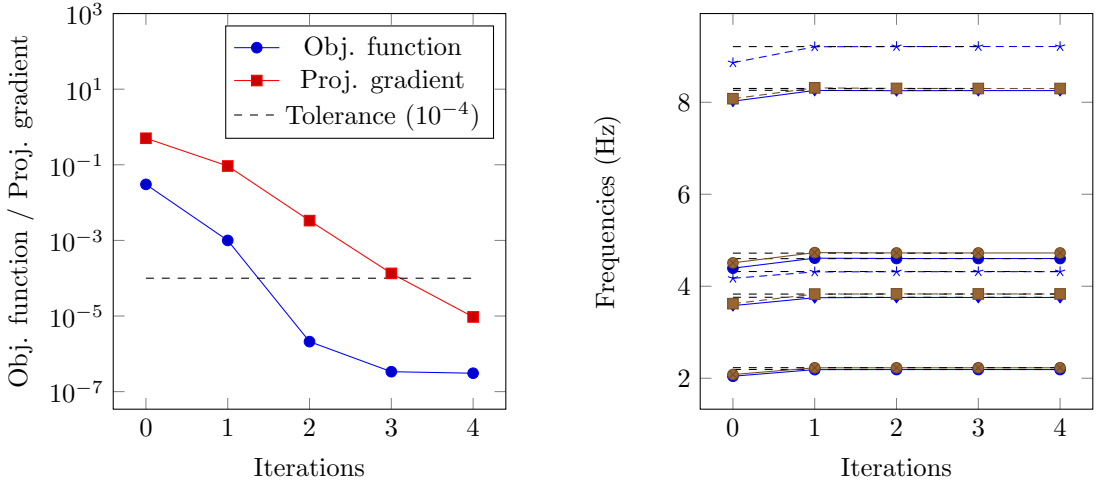

The estimated mechanical parameters obtained through
the optimization are as follows: 
\begin{align*}
E_1 &= 2953\ \mpa & \rho_1 &= 1769\ \kgm &
E_2 &= 3932\ \mpa & \rho_2 &= 1960\ \kgm \\
E_3 &= 3426\ \mpa & \rho_3 &= 1900\ \kgm &
E_4 &= 4951\ \mpa & \rho_2 &= 2174\ \kgm. 
\end{align*}
This corresponds to a relative error between $0.97\%$ and $2.1\%$, with
an average of about $1.6\%$. Although such accuracy is a quite good accuracy for practical purposes,  even more 
precise estimates can be obtained by reducing the tolerance to below
$10^{-4}$ that we have used in these tests. 

\subsection{The Clock tower}

As a real example, we consider the case study of the Clock tower
(``Torre delle Ore'') in Lucca, Italy. The 48.4 m high masonry
structure, dating back to the 13th century, has a rectangular cross
section of about 5.1$\times$7.1 meters and walls of thickness varying
from 1.77 m at the base to 0.85 m at the top. Two barrel vaults are
set at heights of  about 12.5 m and 42.3 m. The bell
chamber is covered by a pavillon roof made up of wooden trusses and
rafters. At a height of about 33 m the walls have been instrumented with 4
steel tie rods of 30$\times$30 mm rectangular section. The adjacent
buildings abut the tower on two sides for a height of about 13 m and
constitute asymmetric boundary conditions. The modal behavior of the
tower has already been analyzed using the NOSA-ITACA code and
through experimental measurements conducted in November 2016, when
the tower was fitted with 4 triaxial seismometric stations. The two
procedures yielded similar, albeit not perfectly matching results
\cite{pellegrini2017anew}. Figure~\ref{fig:tore} shows a picture of the tower,
together with the mesh
used for its discretization.

\begin{figure}
	\centering
	\begin{minipage}{.45\linewidth}
    	\centering
       \includegraphics[width = .35\textheight, angle = 90]{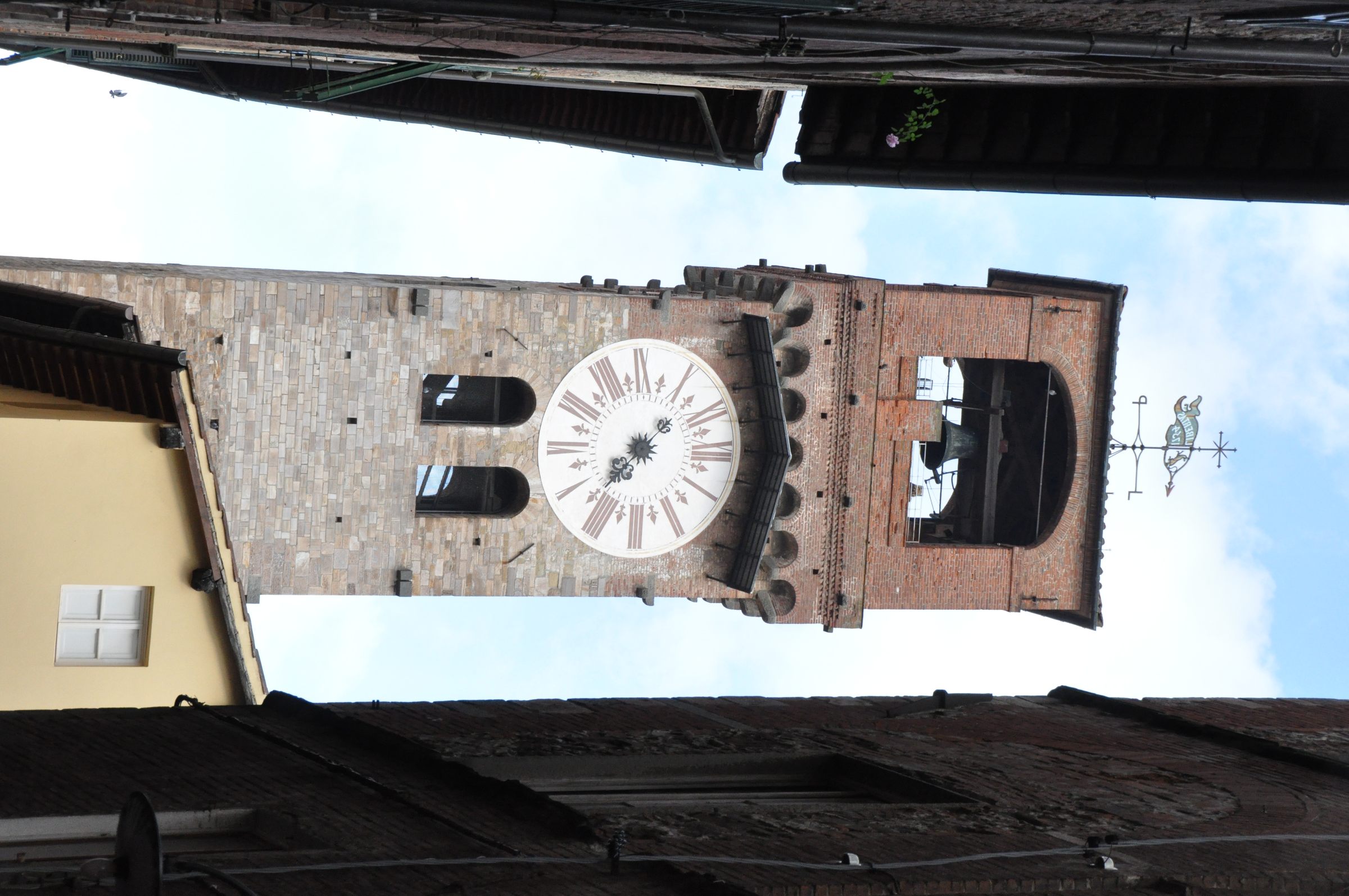}
	\end{minipage}~~\begin{minipage}{.5\linewidth}
	   \centering
       \includegraphics[height = .35\textheight]{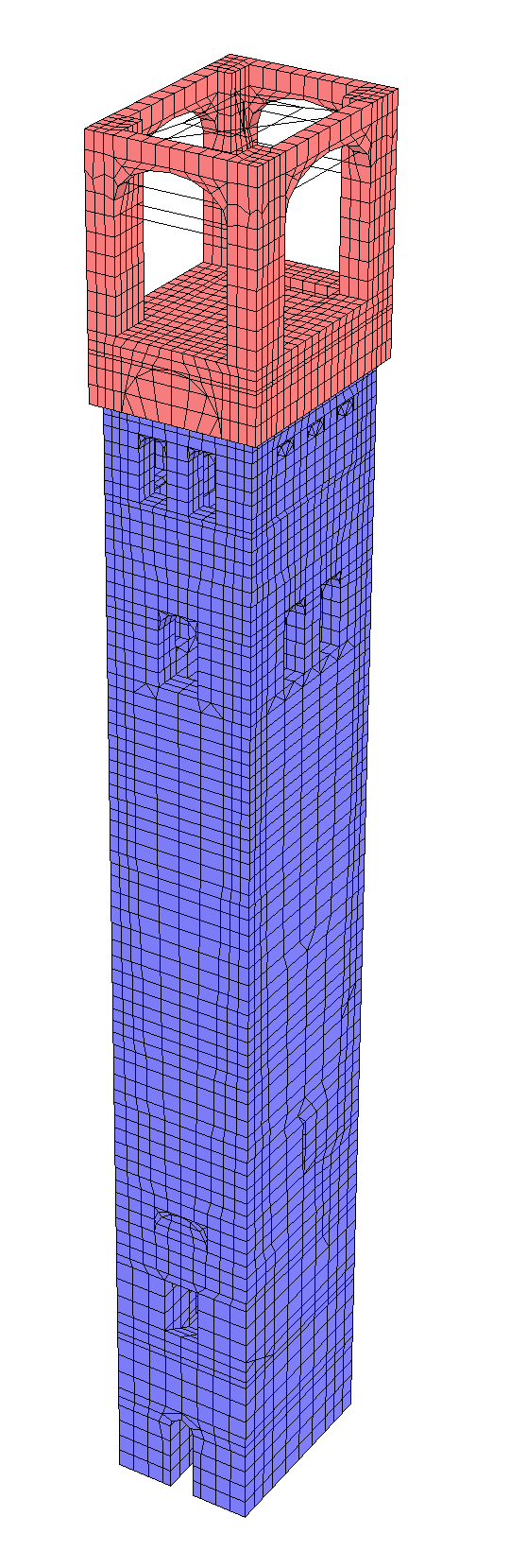}
     \end{minipage}
     \caption{On the left, a picture of the clock tower in Lucca. On the right, 
     	the mesh used for the finite element discretization, where the
     	different materials for the bell chamber and the lower part of the
     	tower are reported in different colors (respectively red and light blue).}
     \label{fig:tore}
\end{figure}

The frequencies obtained by analyzing data from the fitted instruments
via Operational Modal Analysis techniques, are as follows:
\[
  \mathbf f \approx \begin{bmatrix}
    1.05 & 1.3 & 4.19 & 4.50
  \end{bmatrix}.
\]
The frequencies obtained with the
finite element model in \cite{pellegrini2017anew}
are, instead, $[ 0.98 \ 1.24 \ 4.32 \ 4.38 ]$.
The materials constituting
the tower are wood, steel, and masonry, with the latter being different
in the bell chamber and the lower part of the tower.
The mechanical properties used in the finite 
element model are
\begin{align*}
   E_{w} &= 10000\ \mathrm{MPa} & \rho_{w} &= 800\ \mathrm{\frac{kg}{m^{3}}} & \nu_{w} &= 0.35 \\
   \rho_{m_l} &= 2100\ \kgm & \rho_{m_c} &= 1700\ \kgm & \nu_{m_l} &= \nu_{m_c} = 0.2 \\
   E_{s} &= \num{210000}\ \mpa & \rho_s &= 7850\ \kgm & \nu_s &= 0.3, 
\end{align*}
and we leave the Young's moduli $E_{m_c}$ and $E_{m_l}$ as free
parameters, since they are not known explicitly
due to the lack of experimental information. 
Here, the subscript $m_l$ identifies the lower part 
of the tower, and $m_c$ the bell chamber. 

We can reformulate the problem as an optimization one, trying
to match the frequencies of the system to the measured ones,
allowing the parameters to vary in a reasonable range. In our case we
can set
\[
  2500\ \mathrm{MPa} \leq E_{m_l} \leq 5500\ \mathrm{MPa} \qquad
  1000\ \mpa \leq E_{m_c} \leq 5500\ \mpa.
\]
As a first test, we choose the weights as the vector of all ones. 
The convergence of the optimization
scheme, obtained with these initial parameters, is reported in Figure~\ref{fig:convergence-tore}. Note that
the method optimizes the third and fourth frequencies to match quite closely, 
but at the price of decreasing the quality of the match for the first two (see Table~\ref{tab:tore-relative-vs-absolute}). 

\begin{figure} \centering
    \begin{tikzpicture}
    \begin{semilogyaxis}[
    xtick={0,1,2,3,4,5},
    width=.45\linewidth, height = .3\textheight,
    ylabel = Obj. function / Proj. gradient, xlabel = Iterations,
    legend pos = north east, ymax = 1e4]
    \addplot table[x index = 0, y index = 1] {dat/convergence_tore.dat};
    \addplot table[x index = 0, y index = 6] {dat/convergence_tore.dat};
    \addplot[dashed, domain = 0:5] {1e-4};
    \legend{Obj. function, Proj. gradient, Tolerance ($10^{-4}$)}
    \end{semilogyaxis}
    \end{tikzpicture}~~~~~~~~
    \begin{tikzpicture}
    \begin{axis}[width=.45\linewidth, height = .3\textheight,
    ylabel = Frequencies (Hz), xlabel = Iterations, xtick = {0,1,2,3,4,5}]
    \addplot table[x index = 0, y index = 2] {dat/convergence_tore.dat};
    \addplot[domain = 0 : 5, dashed]{ 1.05 };
    \addplot table[x index = 0, y index = 3] {dat/convergence_tore.dat};
    \addplot[domain = 0 : 5, dashed]{ 1.30 };
    \addplot table[x index = 0, y index = 4] {dat/convergence_tore.dat};
    \addplot[domain = 0 : 5, dashed]{ 4.19 };
    \addplot table[x index = 0, y index = 5] {dat/convergence_tore.dat};
    \addplot[domain = 0 : 5, dashed]{ 4.50 };
    \end{axis}
    \end{tikzpicture}
    \caption{The left graph reports the convergence history
        of the objective function
        to the minimum for each new reduced model generated in the
        optimization procedure. The norm of the projected gradient
        is reported as well, and the dashed line is the tolerance of the
        optimization scheme. The right plot represents the convergence of the
        frequencies during the process (the dashed lines are the experimental
        frequencies). Iteration $0$ is the evaluation
        function at the starting point. In this example, the
        free parameters are $E_{m_l}$ and $E_{m_c}$.}
    \label{fig:convergence-tore}
\end{figure}
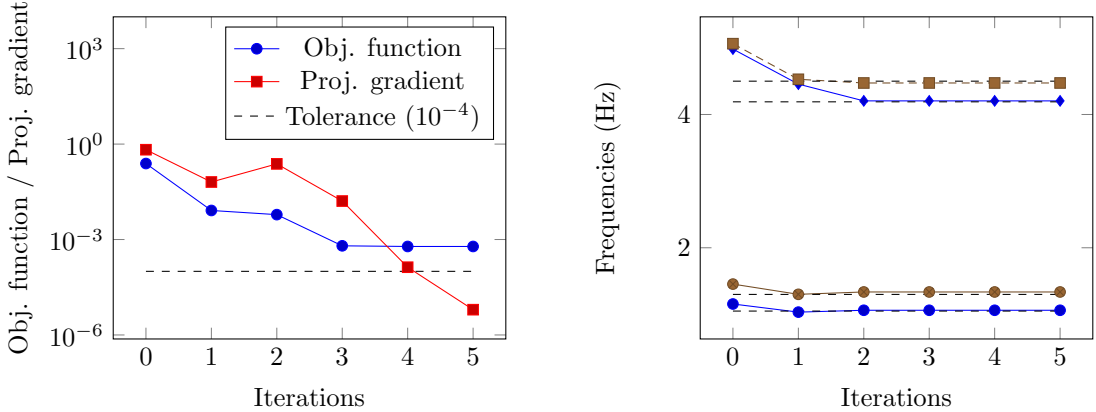

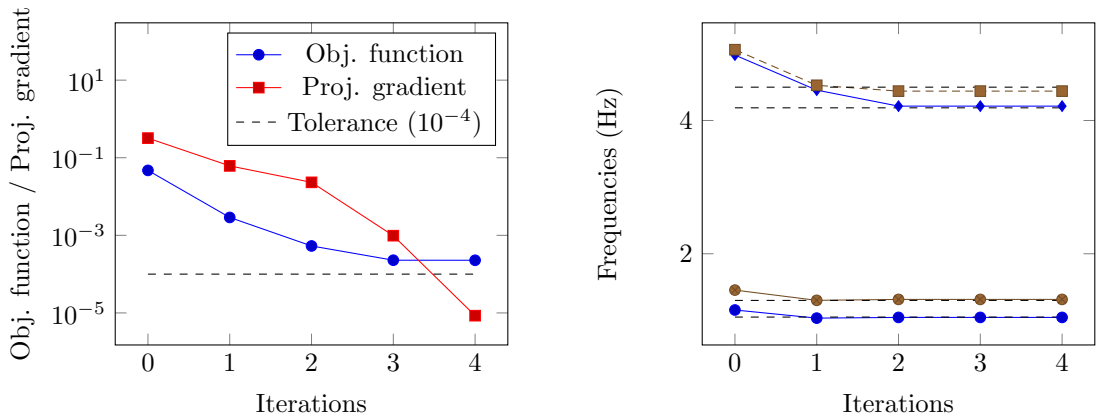
\begin{figure} \centering
	\begin{tikzpicture}
	\begin{semilogyaxis}[width=.45\linewidth, height = .3\textheight,
	ylabel = Obj. function / Proj. gradient, xlabel = Iterations, legend pos = north east,
	ymax = 300]
	\addplot table {dat/convergence_tore_relative.dat};
	\addplot table[x index = 0, y index = 6] {dat/convergence_tore_relative.dat};
	\addplot[dashed, domain = 0:4] {1e-4};
	\legend{Obj. function, Proj. gradient, Tolerance ($10^{-4}$)}
	\end{semilogyaxis}
	\end{tikzpicture}~~~~~~~~
	\begin{tikzpicture}
	\begin{axis}[width=.45\linewidth, height = .3\textheight,
	ylabel = Frequencies (Hz), xlabel = Iterations]
	\addplot table[x index = 0, y index = 2] {dat/convergence_tore_relative.dat};
	\addplot[domain = 0 : 4, dashed]{ 1.05 };
	\addplot table[x index = 0, y index = 3] {dat/convergence_tore_relative.dat};
	\addplot[domain = 0 : 4, dashed]{ 1.30 };
	\addplot table[x index = 0, y index = 4] {dat/convergence_tore_relative.dat};
	\addplot[domain = 0 : 4, dashed]{ 4.19 };
	\addplot table[x index = 0, y index = 5] {dat/convergence_tore_relative.dat};
	\addplot[domain = 0 : 4, dashed]{ 4.50 };
	\end{axis}
	\end{tikzpicture}
	\caption{The left graph reports the convergence history
		of the objective function
		to the minimum for each new reduced model generated in the
		optimization procedure. The norm of the projected gradient
		is reported as well, and the dashed line is the tolerance
		set for the projected gradient in the optimization scheme.
		The right plot represents the convergence of the
		frequencies during the process (the dashed line are the experimental
		frequencies). Iteration $0$ is the evaluation
		function at the starting point. In this example, the
		free parameters are  $E_{m_l}$ and $E_{m_c}$, and the objective
		function is defined using the weights $w_i = f_i^{-1}$. }
	\label{fig:convergence-tore-relative}
\end{figure}

This is connected to the choice of $\mathbf w$, which ensures
norm-wise accuracy, but not a relative one. Often, we prefer to have a low
\emph{relative} error on the frequencies, instead of an absolute one.
Therefore, it is natural to choose $w_i = f_i^{-1}$ (in fact, this
is the default choice in the implementation). This yields
different results, and different approximations for both the parameters
and the frequencies. More specifically, we expect a better match
on the lowest end of the spectrum, and a worse one on the highest
frequencies.

The convergence history with this choice of $\mathbf w$ is reported in Figure~\ref{fig:convergence-tore-relative}, and the results
are shown in Table~\ref{tab:tore-relative-vs-absolute}. It appears
immediately evident that the relative errors are balanced with $w_i = f_i^{-1}$,
whereas they tend to increase on the small frequencies when $w_i = 1$.

The parameters identified by the model updating phase give 
$E_{m_l} \approx \num{3182}\ \mpa$ and $E_{m_c} \approx \num{1873}\ \mpa$ when using
weights equal to one, and $E_{m_l} \approx \num{3076}\ \mpa$ and $E_{m_c} \approx \num{1950}\ \mpa$ when aiming for relative accuracy. 

We conclude this section by comparing the overall CPU time of 
our implementation with those of alternative
solvers, which in various ways fail to exploit the problem structure.
Table~\ref{tab:timings} compares 4 different implementations:
\begin{description}
   \item[RM] This implementation corresponds to the current proposal, that is,
     a solver based on the trust-region and
     projection approach described in the previous section and used in the 
     tests presented herein. This strategy is labeled as ``RM'' in the table --- highlighted with the bold font.
   \item[BB] Here we interpret the finite element code as a black-box function and solve
   the optimization problem using a general purpose optimizer.
     This is done to reflect a scenario in which the user has no access to the
     internals of the finite element code --- and requires an assembly phase at
     every function evaluation. Derivatives are computed here
     by finite differences. This is column ``BB'' in the table.
   \item[A] In this approach, we separately assemble the parametric generalized eigenvalue
     problem $K(\mathbf x) - \lambda M(\mathbf x)$ and then pass to a general 
     optimizer a function that evaluates the frequencies for a given value
     of $\mathbf x$. This function corresponds to the Lanczos projection 
     routine used internally in RM.  Derivatives are computed here
     by finite differences. This is approach ``A'' in the table.
   \item[AD] The same as A, but derivatives are evaluated analytically
     as described to Section~\ref{sec:derivatives}. This is ``AD'' in the table.

\end{description}
The general optimizer chosen for the BB, A and AD implementations is again 
{\tt fmincon} with the \texttt{sqp} option available in the Optimization toolbox of MATLAB.
In our experience, this choice turns out to be the most efficient 
among those available in the toolbox. This same function
has been used for 
computation of the step in the trust-region scheme of Algorithm \ref{alg:tr}.

In Table~\ref{tab:timings}, the time for assembly (which needs to be added to the costs in columns
A, AD, and RM) is reported separately, as well as the problem's degrees of freedom. 

It is immediately apparent that the proposed RM method is the fastest in all cases
except the arch on piers, where given the small dimension incurs in some overhead.
Moreover, it is worth stressing that our approach exhibits a complexity that
seems to grow nicely with the dimension. The general optimization approach BB, instead,
becomes costly when the number of degrees of freedom increases.

For instance, in the dome example, the timings range from about 7 hours for BB 
to slightly more than 3 minutes with our approach. The
 case where we supply derivatives to the optimizer (AD) 
still takes more than 30 minutes.

\begin{table}
    \begin{tabular}{c|cc|cc}
      Exp. freq. & Freq. ($w_i = 1$) & Rel. error ($w_i = 1)$ & Freq. ($w_i = f_i^{-1}$) & Rel. error ($w_i = f_i^{-1}$) \\ \hline
      $1.05$ Hz & $1.0621$ Hz & $1.15$\% & $1.0449$ Hz & $0.49$\% \\
      $1.30$ Hz & $1.3366$ Hz & $2.82$\% & $1.315$ Hz & $1.15$\% \\
      $4.19$ Hz & $4.2041$ Hz & $0.33$\% & $4.2154$ Hz & $0.61$\% \\
      $4.50$ Hz & $4.4729$ Hz & $0.60$\% & $4.4409$ Hz  & $1.31$\% \\
    \end{tabular}
    \caption{Frequencies and corresponding relative errors computed
        using the weights $w_i = 1$ and $w_i = f_i^{-1}$ for
        the problem of the clock tower in Lucca, Italy.}
    \label{tab:tore-relative-vs-absolute}
\end{table}

\begin{table}
    \centering
    \begin{tabular}{c|cccccc}
               & \# dof& BB & A & AD & \textbf{RM} & Assembly \\ \hline
        Arch on piers   & 851 & 16.40s    & 1.17s     & 0.62s  & \textbf{0.68s}   & 0.42s \\
        Dome        & 122,853 & 25,032s & 7,662.4s  & 1,990.5s & \textbf{226.02s} & 134.32s \\
        Clock tower & 45,511 & 476.9s   & 305.39s   & 202.87s & \textbf{44.10s}  & 32.22s \\
    \end{tabular}
    \caption{Timings for different optimization approaches. The timings
        reported for the approaches ``Assembled'' (A), ``Assembled $+$ derivatives'' (AD), and
        ``Reduced model'' (RM) do not include the time needed to assembly the
        parametric finite element model, which is reported in the
        last column. The Portal model has been run with the default
        starting points. }
    \label{tab:timings}
\end{table}

The timings suggest the unsurprising conclusion that better knowledge of the underlying
optimization problem leads to more efficient and faster optimization routines.
It should be noted that 
the stopping criterion for all the approaches is set to ensure that the first order
optimality condition (given by the norm of the projected gradient) is smaller than 
$\epsilon = 10^{-4}$ used in the experiments. All the approaches provide solutions
with comparable accuracy.

\section{Concluding remarks}
\label{sec:concluding-remarks}
In this work we have presented a model updating scheme focusing on the
optimization of finite element models for structural dynamics. We have shown
how the subspace projection method used to compute the eigenvalues at the
lowest end of the spectrum can be used directly as a parametric model
order reduction step, and that embedding it in a trust-region scheme
can be very effective. 

Several examples have been reported on, both to test the theoretical properties
of the scheme as well as to demonstrate its practical applicability. We have
also presented a preliminary numerical evidence of its robustness against perturbations in the data, an aspect which will be further investigated and characterized
in future work. 

We believe that embedding the model updating step directly in the finite
element code naturally leads to more efficient procedures, compared to applying
an optimizer fed with a black-box finite element code. The proposed approach can benefit
from the fact that the optimizer knows the details about the finite element
formulation, and this reveals to be particularly effective in terms 
of reliability and efficiency. 

Several problems remain open, and deserve further study in the future. For instance, optimization of mode shapes --- and 
not only eigenfrequencies --- is of crucial importance to obtain an
accurate understanding of the dynamic behavior of a structure. 

Moreover, a robust and efficient implementation of the approach within 
the NOSA-ITACA code needs to be performed --- along with its integration with the
graphical user interface based on SALOME. These issues will be addressed in future work.

\section*{References}

\bibliographystyle{elsarticle-harv}
\bibliography{paper}

\end{document}